\documentclass[10pt]{amsart}
\usepackage{amsmath,amsthm,amsfonts,amssymb}
\usepackage{hyperref}
\usepackage[all]{xy}
\usepackage{graphicx}
\usepackage{amscd}
\usepackage{graphicx}
\usepackage{pb-diagram}
\usepackage{color}
\usepackage{kotex}
\usepackage{tikz, caption, subcaption}
\usepackage{enumitem}
\usepackage[utf8]{inputenc}

\usetikzlibrary{positioning, fit, calc}   
\tikzset{block/.style={draw, thick, text width=3cm ,minimum height=1.3cm, align=center},   
line/.style={-latex}     
}  

\numberwithin{equation}{section}

\newtheorem{theorem}{Theorem}[section]

\newtheorem{corollary}[theorem]{Corollary}

\newtheorem{lemma}[theorem]{Lemma}

\newtheorem{prop}[theorem]{Proposition}

\newcommand{\R}{{\mathbb R}}
\newcommand{\C}{\mathbb C}
\newcommand{\Z}{{\mathbb Z}}

\newcommand{\CF}{{\mathcal F}}

\newcommand{\CI}{{\mathcal I}}
\newcommand{\CJ}{{\mathcal J}}

\newcommand{\CT}{{\mathcal T}}

\newcommand{\ep}{{\varepsilon}}
\newcommand{\wt}{\widetilde}
\newcommand{\wh}{\widehat}
\newcommand{\supp}{\mathrm{supp\,}}

\newcommand{\IM}{\mathrm{Im\,}}
\newcommand{\RE}{\mathrm{Re\,}}

\newcommand{\delc}{{\delta_\circ}}

\begin{document}

\author[E. Jeong, Y. Kwon, and S. Lee]{Eunhee Jeong, Yehyun Kwon, and Sanghyuk Lee}

\subjclass[2010]{42B15; 35B60}
\keywords{Carleman inequality, unique continuation, polyharmonic operator}

\address{(Eunhee Jeong)  Department of Mathematics Education and  Institute of Pure and Applied Mathematics, Jeonbuk National University, Jeonju 54896, Republic of Korea}
\email{eunhee@jbnu.ac.kr}

\address{(Yehyun Kwon) School of Mathematics, Korea Institute for Advanced Study, Seoul 02455, Republic of Korea}
\email{yhkwon@kias.re.kr}

\address{(Sanghyuk Lee) Department of Mathematical Sciences, Seoul National University, Seoul 151-747, Republic of Korea}
\email{shklee@snu.ac.kr}

\title[Carleman inequalities and unique continuation]{Carleman inequalities and unique continuation for the polyharmonic operators}
\maketitle

\begin{abstract}
We obtain a complete characterization of $L^p-L^q$ Carleman estimates with weight $e^{v\cdot x}$ for the polyharmonic operators. Our result extends the Carleman inequalities for the Laplacian due to Kenig--Ruiz--Sogge. Consequently, we obtain new unique continuation properties of  higher order Schr\"odinger equations relaxing the integrability assumption on the solution spaces.
\end{abstract}

\section{Introduction}
Let $d$ and $k$ be positive integers and $\Omega$ a non-empty connected open set in $\R^d$. We denote by $V$ a complex-valued function defined on $\Omega$ of which precise description to be given below. We say that the differential inequality 
\begin{equation}\label{e:diff_ineq}
	|\Delta^k u | \le |Vu| \ \ \text{in} \ \ \Omega
\end{equation}
has the \emph{unique continuation property} (UCP) in the Sobolev space $W^{2k,p}_{\mathrm{loc}}(\Omega)$ if every $u\in W^{2k,p}_{\mathrm{loc}}(\Omega)$ satisfying \eqref{e:diff_ineq} and vanishing on a non-empty open subset of $\Omega$ is identically zero in $\Omega$.  

When $k=1$ and $d\ge3$, the inequality \eqref{e:diff_ineq} includes the Schr\"odinger equation $(-\Delta+V)u=0$ as a special case, of which UCP has been extensively studied  over the past several decades by numerous authors. We refer the reader to the survey articles \cite{Ken89, W93, KT01} and \cite{Ler19}, and references therein. For $V\in L_{\mathrm{loc}}^{d/2}(\Omega)$,  Kenig,  Ruiz, and Sogge \cite{KRS87} proved that the inequality
\begin{equation}\label{e:diff-ineq-1}
	|\Delta u | \le |Vu| \ \ \text{in} \ \ \Omega
\end{equation}
has the UCP in $W_{\mathrm{loc}}^{2,p}(\Omega)$ whenever $p>\frac{2d}{d+3}$. The main ingredient in their argument is the Carleman inequality of the form 
\begin{equation}\label{e:carl-1}
\| e^{v\cdot x} u\|_{L^q(\R^d)} \le C \|e^{v\cdot x} \Delta u\|_{L^p(\R^d)}
\end{equation} 
where $C$ is a constant independent of $v\in \R^d$ and $u\in C^\infty_0(\R^d)$. For $p,q$ satisfying 
\begin{equation}\label{e:cond-1}
\tfrac1p-\tfrac1q=\tfrac 2d \ \ \text{and} \ \  \tfrac{d+1}{2d}<\tfrac1p< \tfrac{d+3}{2d},
\end{equation}
they deduced \eqref{e:carl-1} from the more general \emph{uniform Sobolev inequality} 
\begin{equation}\label{e:unisob}
\| u\|_{L^q(\R^d)} \le C \|(-\Delta+a\cdot\nabla+b) u\|_{L^p(\R^d)}, \ \ \forall (a,b)\in \C^d\times \C,
\end{equation}
which holds if and only if \eqref{e:cond-1} is satisfied. 

In \cite{jkl18} the authors characterized the full range of $(p,q)$ on which the Carleman inequality \eqref{e:carl-1} holds. More precisely, under the assumption $1<p,q<\infty$, it was proved that \eqref{e:carl-1} holds if and only if  
\[	\tfrac1p-\tfrac1q=\tfrac 2d \ \ \text{and} \ \  \tfrac{d^2-4}{2d(d-1)}\le \tfrac1p \le \tfrac{d+2}{2(d-1)}.	\]
This range is larger than that of \eqref{e:cond-1} whenever $d\ge 4$. Consequently, the integrability assumption on the solution space in which the differential inequality \eqref{e:diff-ineq-1} has the UCP can be relaxed. That is to say, \eqref{e:diff-ineq-1} has the UCP in $\bigcup_{p>1} W^{2,p}_{\mathrm{loc}}(\Omega)$ if $d=3$ or $4$, and in $W^{2,\frac{2(d-1)}{d+2}}_{\mathrm{loc}}(\Omega)$ if $d\ge5$. 

The polyharmonic operators $\Delta^k$, $k\ge2$, are prototypical examples of higher order elliptic operators and have abundant applications in various physical contexts (\cite{Mel03, GGS}). In this paper, we aim to study the UCP of the differential inequality \eqref{e:diff_ineq}, which is a natural higher order analogue of \eqref{e:diff-ineq-1}. The UCP of \eqref{e:diff_ineq} is closely tied to the Carleman inequality
\begin{equation}\label{carl}
\| e^{v\cdot x} u\|_{L^q(\R^d)} \le C \|e^{v\cdot x} (-\Delta)^k  u\|_{L^p(\R^d)}
\end{equation} 
with a constant $C$, independent of  $v\in\R^d$ and $u\in C_0^\infty(\R^d)$.  

If $p=\frac{2d}{d+2k}$ and $q=p'=\frac{2d}{d-2k}$, then \eqref{carl} can be obtained from the Carleman inequality with weights $|x|^{-\tau}$ proved in \cite{JK85, La88}, as was observed by Wolff \cite{W93} for the case $k=1$. See \cite[Proposition 2.2]{KU16}. 

In this paper, we completely characterize the range of $p$ and $q$,  on which the Carleman inequality \eqref{carl} holds.

\begin{theorem}\label{t:carl}
Let $1<p, q <\infty$. The  inequality \eqref{carl} holds if and only if \begin{equation}\label{e:cond}
	\tfrac1p- \tfrac1q= \tfrac{2k}d \ \ \text{and} \ \ \tfrac{(d+2k)(d-2)}{2d(d-1)} \le \tfrac1p \le \tfrac{d+2k}{2(d-1)}.
\end{equation}
\end{theorem}
If $\frac{d-2}2\le k<\frac d2$, then the second condition in \eqref{e:cond} is vacuously true provided that $1<p,q<\infty$ and $\frac1p-\frac1q=\frac{2k}d$. Hence, in this case, the inequality \eqref{carl} holds for all $1<p,q<\infty$ satisfying the gap condition $\frac1p-\frac1q=\frac{2k}d$. For $k<\frac{d-2}{2}$, however, the second condition in  \eqref{e:cond} is nontrivial.\footnote{See, for example, Figures \ref{fig1} and \ref{fig2} in Section \ref{s:c_wucp}, where the necessary and sufficient condition for \eqref{carl} is described by the thick line segments.} For $k\ge \frac d2$, \eqref{carl} is not true because the set \eqref{e:cond} is empty. 

By a standard density argument the Carleman inequality \eqref{carl} is also valid for all $u\in W^{2k,p}_0(\R^d)$ for every $p$ given as in Theorem \ref{t:carl}. We shall use this fact in Section \ref{s:wucp} to deduce a unique continuation result (Theorem \ref{t:wucp}) from Theorem \ref{t:carl}.

In order to prove Theorem \ref{t:carl}, we basically follow the strategy of the authors' previous work \cite{jkl18}. This reduces the Carleman inequality \eqref{carl} to obtaining $L^p-L^q$ boundedness of a Fourier multiplier operator of which singularity is on the $d-2$ dimensional unit sphere $S^{d-2}$ embedded in $\R^d$. We dyadically decompose the multiplier near the sphere $S^{d-2}$, and obtain sharp bounds for the dyadic pieces. In this approach, the Fourier restriction-extension operator (see Section \ref{sec:har_anal}) defined by $S^{d-2}$ naturally arises if we write the Fourier multiplier in the cylinderical coordinates in $\R^d$. 

 However, direct use of  the bounds on the restriction-extension operator as in \cite{jkl18} does not yield any sharp estimate when $k\ge2$. In such cases, the associated multiplier $m$ can be regarded to have singularity of degree $k$ on $S^{d-2}\times \{0\}$. So, the bounds on the Fourier restriction-extension operator, whose multiplier is of order $-1$, is not enough to handle the operator defined by $m$. Our novelty in this paper is in overcoming the difficulty by making use of  the Bochner--Riesz operators of indices less than $-1$ (Lemma \ref{lem_BR}). However, the Bochner--Riesz operators become  more singular. This necessitates manipulation of the associated distribution by raising and reducing the order of operators. See Lemmas \ref{l:order-increasing} and \ref{l:counter}, and \emph{Proof of Proposition \ref{p:tilde}}.

If we make use of Theorem \ref{t:carl} and adjust the argument in \cite{KRS87}, then we obtain the following unique continuation result.

\begin{theorem}\label{t:wucp}
Let $1\le k< \frac d2$, $\Omega$ a non-empty connected open set in $\R^d$, and let
\[	{\mathbf X}={\mathbf X}_{k,d}(\Omega):=
	\begin{cases}
	\bigcup_{p>1} W_{\mathrm{loc}}^{2k, p}(\Omega)  &\text{if} \ \  k\ge \frac{d-2}2, \\
	W_{\mathrm{loc}}^{2k, \frac{2(d-1)}{d+2k}}(\Omega) & \text{if} \ \  k<\frac{d-2}2.
	\end{cases}	\]
Then, for every $V\in L_{\mathrm {loc}}^{d/2k}(\Omega)$ the differential inequality \eqref{e:diff_ineq} has the UCP in $\mathbf X$.
\end{theorem}
When $k=1$, Koch and Tataru \cite{KT02} proved that the potential class $L_{\mathrm{loc}}^{d/2}$ is critical in the scale of Lebesgue spaces for \eqref{e:diff-ineq-1} to have the UCP; for every $r<d/2$ they constructed a nontrivial smooth function $u\in C_0^\infty(\R^d)$ such that $\Delta u/u \in L^r(\R^d)$. By following their construction, it is possible to show that for every $k\in[2, d/2)$ the $L_{\mathrm{loc}}^{d/2k}$-potentials are critical for \eqref{e:diff_ineq} to have the UCP in $C^\infty_0(\R^d)$ (\cite{JK22}). 
 
On the other hand, the problem of finding the largest possible solution space where the UCP holds is an interesting problem. That is, given a $V\in L_{\mathrm{loc}}^{d/2k}(\R^d)$, one can ask whether or not the integrability of the solution space $\mathbf X$ in Theorem \ref{t:wucp} is optimal for the UCP of \eqref{e:diff_ineq}. Although we obtain Theorem \ref{t:wucp} by means of the Carleman inequality \eqref{carl}, which is completely characterized in Theorem \ref{t:carl}, the condition \eqref{e:cond} itself does not give the sharpness of the integrability assumption of $\mathbf X$ in Theorem \ref{t:wucp}. Even when $k=1$, it is an open problem to find the smallest $p^\ast$ such that \eqref{e:diff-ineq-1} with $V\in L^{d/2}_{\mathrm{loc}}(\Omega)$ has the UCP in $W_{\mathrm{loc}}^{2,p}(\Omega)$ if $p>p^\ast$. Also, it seems an interesting open problem to consider the UCP of \eqref{e:diff_ineq} when $k\ge d/2$.

\subsubsection*{Organization} 
In Section \ref{sec:har_anal} we present some tools from harmonic analysis, namely, the estimates for the Bochner--Riesz operators of negative indices and their consequences. These are crucial in our argument obtaining the complete set of the Lebesgue exponents for which \eqref{carl} holds. In Section \ref{s:c_wucp} we prove Theorem \ref{t:carl}, and in Section \ref{s:wucp} we prove Theorem \ref{t:wucp}.  As a rigorous proof of an estimate which yields the necessity part in Theorem \ref{t:carl} is somewhat involved, we postpone the proof until the last section.

\subsubsection*{Notations}
For a set $A\subset \R^n$ we denote by $C_0^\infty(A)$ the class of smooth functions each of which can be extended to a smooth function defined on some open set $U$ containing $A$ and supported in $A$. Denoting by $\alpha\in \mathbb N_0^d$  a multi-index  we use the standard notation  
\[	W^{m,p}(\Omega)=\{f\in L^p(\Omega) \colon \|\partial^\alpha f \|_{L^p(\Omega)} <\infty, \, |\alpha|\le m \}	\]
for the usual Sobolev space. Moreover, $W^{m,p}_0(\Omega)$ denotes the closure of $C_0^\infty(\Omega)$ in $W^{m,p}(\Omega)$.  For $1\le p,q\le \infty$, we define
\[	\|T\|_{p\to q}=\sup\{\|Tf\|_{L^q(\R^d)}\colon f\in \mathcal S(\R^d), \ \|f\|_{L^p(\R^d)}=1 \}	\] 
for a linear operator  $T$ acting on the Schwartz class $\mathcal S(\mathbb R^d)$.  For the Fourier transform we follow the convention $\wh f(\xi)= \CF f(\xi) = \int_{\R^d} e^{-ix\cdot\xi} f(x)dx$, so the inverse Fourier transform is defined by $f^\vee(x)=\CF^{-1}f(x)=(2\pi)^{-d}\CF f(-x)$. For a bounded function $m$ we denote by $m(D)$ the associated Fourier multiplier operator, that is, $m(D)f=\CF^{-1}(m\CF f)$. We slightly abuse notation so that the dimension of the Fourier transform and its inversion may vary depending on context.

\section{The Bochner--Riesz operator of negative index}\label{sec:har_anal}
Let us recall from \cite[Chapter III]{Hor-book} the analytic family of distributions $\chi_+^a \in \mathcal D'(\R)$ defined by
\[	\chi_+^a = \frac{x_+^a}{\Gamma(a+1)}, \ \   \mathrm{Re\,}a>-1.	\]
$\chi_+^a$ can be continued analytically to all $a\in \C$ so that $d\chi_+^a /dx =\chi_+^{a-1}$. Since $\chi_+^0$ is the Heaviside function it follows that $\chi_+^{-k}=\delta_0^{(k-1)}$ for $k\in \mathbb N$.  

Let us denote by $\lambda$ the smooth function from $\R^n\setminus\{0\}$, $n\ge2$,  defined by $\lambda(\theta)=1-|\theta|^2.$	Then the Bochner--Riesz operator $T_\alpha$ on $\R^n$ of (negative) index $-\alpha\in [-\frac{n+1}2,0)$ is the Fourier multiplier operator defined by 
\[	T_\alpha f = \CF^{-1} \big( (\lambda^\ast \chi_+^{-\alpha} ) \CF f\big), \]
where $\lambda^\ast \chi_+^{-\alpha} \in \mathcal D'(\R^n\setminus\{0\})$ is the pullback of $\chi_+^{-\alpha}$ via the function $\lambda$ (see \cite[Chapter VI]{Hor-book}). In particular, denoting by $d\sigma_{S^{n-1}}$ the surface measure on the unit sphere $S^{n-1}$ we have
\begin{equation}\label{e:sp_msr}
	\lambda^\ast \chi_+^{-1}= \lambda^\ast \delta_0 = c\, d\sigma_{S^{n-1}}
\end{equation}
for some constant $c$ depending on $n$. Hence  $T_1$ is the Fourier restriction-extension operator on the sphere $S^{n-1}$. 

The $L^p - L^q$ boundedness of $T_\alpha$ with $\alpha\in (0,\frac{n+1}2]$ was studied by several authors (\cite{Bor86, CS88, BMO, Ba97, Gu00, CKLS, KL19}). We will make use of it to prove the Carleman inequality \eqref{carl} for all admissible $p,q$ in the following section. Here, let us for the moment digress and introduce some convenient notations to describe the set of $p,q$ for which $\|T_\alpha\|_{p\to q}$ is bounded. 

We set $\mathbf{Q} =[0, 1]\times [0,1]\subset \R^2$, and for $(x,y) \in \mathbf{Q}$ we define $(x,y)' =(1-y, 1-x)$. Similarly, for $\mathcal R\subset \mathbf{Q}$ we denote $\mathcal R' =\{ (x,y)\in \mathbf{Q} \colon (x,y)' \in \mathcal R \}$. For $X_1, \ldots, X_m\in \mathbf{Q}$, we use the notation $[X_1, \ldots, X_m]$ to denote the convex hull of the points $X_1, \ldots, X_m$. In particular, if  $X, Y\in \mathbf{Q}$,  $[X,Y]$  is  the closed line segment  connecting $X$ and $Y$ in $\mathbf{Q}$. By $(X,Y)$ and $[X,Y)$ we denote the open interval $[X,Y]\setminus\{X,Y\}$ and the half-open interval $[X,Y]\setminus\{Y\}$, respectively. Furthermore, for $d\ge 3$ and $0<\alpha<d/2$, let us set 
\begin{equation}\label{e:ptBD}
	B_\alpha^d = \big(\tfrac{d-2+2\alpha}{2(d-1)}, \tfrac{(d-2)(d-2\alpha)}{2d(d-1)} \big), \ D^d_\alpha=\big(\tfrac{d-2+2\alpha}{2(d-1)}, 0 \big), \  \text{and} \  H=(1,0).
\end{equation}
See Figures \ref{fig1} and \ref{fig2}.

We resume our discussion of the operator $T_\alpha$. It was shown by B\"orjeson \cite{Bor86} that $\|T_\alpha\|_{L^p(\R^n)\to L^q(\R^n)}<\infty$ only if $(\frac1p, \frac1q)$ lies in the set
\begin{align*}
\mathcal P_{\alpha}^{n+1} 
	&:= [B_{\alpha}^{n+1}, (B_{\alpha}^{n+1})' , D_{\alpha}^{n+1}, (D_{\alpha}^{n+1})', H] \setminus \big( [B_{\alpha}^{n+1}, D_{\alpha}^{n+1}] \cup [B_{\alpha}^{n+1}, D_{\alpha}^{n+1}]' \big)\\
	& = \big\{(x,y)\in \mathbf{Q} \colon x-y\ge \tfrac{2\alpha}{n+1}, \, x>\tfrac{n-1+2\alpha}{2n}, \, y < \tfrac{n+1-2\alpha}{2n} \big\}.
\end{align*}
When $n=2$ the sufficiency was proved by Bak \cite{Ba97} for $0<\alpha\le 3/2$. In higher dimensions, when $0<\alpha <\alpha^\ast$ for some $\alpha^\ast<1/2$, the complete characterization of $L^p- L^q$ boundedness of $T_\alpha$ still remains open \cite{BMO, CKLS, KL19}. For more on this problem we refer the reader to \cite{KL19}, where the currently widest range of $\alpha$ can be found. We also mention the recent result \cite{MS21} in which new Bochner--Riesz estimates with negative index associated to non-elliptic surfaces were proven.

In this paper, we particularly make use of the following  estimates  for $T_k$ with $k\in [1,\frac{n+1}2)\cap \mathbb N$, which are due to Guti\'errez \cite{Gu00}, Bak \cite{Ba97}, and Cho--Kim--Lee--Shim \cite{CKLS}:
\begin{align}
	\|T_k f\|_{L^{q}(\R^n)} &\lesssim \|f\|_{L^{p}(\R^n)} \ \ \text{for} \ \  (\tfrac1p,\tfrac1q)\in \mathcal P_k^{n+1}, \label{e:st-BR} \\
	\|T_k f\|_{L^q (\R^n)} &\lesssim \|f\|_{L^{p, 1}(\R^n)} \ \ \text{for} \ \  (\tfrac1p,\tfrac1q)\in (B_{k}^{n+1}, D_{k}^{n+1}],\label{e:endpt-BR1}\\
	\|T_k f\|_{L^{q, \infty}(\R^n)} &\lesssim \|f\|_{L^{p, 1}(\R^n)} \ \ \text{for} \ \  (\tfrac1p,\tfrac1q)=B_{k}^{n+1}.  \label{e:endpt-BR}
\end{align}

In order to prove the Carleman inequality \eqref{carl} for \emph{all} $p,q$ described in Theorem \ref{t:carl}, we make use of the estimates \eqref{e:st-BR}--\eqref{e:endpt-BR}; see Lemma \ref{lem_BR}.  To this end, we relate the multiplier $\lambda^\ast \chi_+^{-k}$ to $\lambda^\ast \chi_+^{-1}$. We use the standard notation $\mathcal E'(\R^n\setminus \{0\})$ denoting the set of distributions compactly supported in $\R^n\setminus \{0\}$. It is clear that $\lambda^\ast \chi_+^{-k}$ is supported on $S^{n-1}$.
\begin{lemma}\label{l:order-increasing}
Let $n\ge 2$ and $k\ge 1$ be integers, and let $\langle\cdot, \cdot\rangle$ be the duality pairing on $\mathcal E'(\R^{n}\setminus \{0\})\times C^\infty(\R^{n}\setminus \{0\})$. Then, for $\phi\in C^\infty(\R^n\setminus\{0\})$ we have
\begin{equation}\label{e:dist-id}
	\langle\lambda^\ast \chi_+^{-k}, \phi \rangle = \langle \lambda^\ast\chi_+^{-1}, L^{k-1}\phi \rangle,
\end{equation}
where $L$ is a differential operator defined by 
\[	L\phi (\theta) = \frac1{2|\theta|^2} ( n-2 + \theta\cdot\nabla)\phi (\theta), \ \  \theta\in \R^n\setminus\{0\}. \]
\end{lemma}
\begin{proof}
It is enough to prove \eqref{e:dist-id} for $k\ge2$. We denote by $\{\partial_j=\frac{\partial}{\partial \theta_j} \colon 1\le j\le n\}$ the standard orthonormal frame on $\R^n$. The chain rule (see, e.g., \cite[p. 135]{Hor-book}) gives 
\[	\sum_{j=1}^n\theta_j \partial_j \lambda^\ast \delta_0^{(k-2)} = \lambda^\ast \delta_0^{(k-1)} \sum_{j=1}^n\theta_j\partial_j\lambda = -2|\theta|^2 \lambda^\ast \delta_0^{(k-1)}.	\]
Hence,
\begin{align*}
\langle \lambda^\ast\chi_+^{-k}, \phi \rangle 
&= \langle \lambda^\ast \delta_0^{(k-1)}, \phi \rangle 
= -\tfrac1{2}\sum_{j=1}^n \big \langle \tfrac{\theta_j}{|\theta|^2} \partial_j \lambda^\ast \delta_0^{(k-2)} , \phi \big \rangle \\
&= \big \langle \lambda^\ast \delta_0^{(k-2)}, \tfrac12 \sum_{j=1}^n \partial_j \big(\tfrac{\theta_j\phi}{|\theta|^2} \big) \big\rangle 
= \langle \lambda^\ast \chi_+^{-(k-1)}, L \phi \rangle.
\end{align*}
Repeating this argument we get the identity \eqref{e:dist-id}. 
\end{proof}

For $\rho>0$ let $\lambda_\rho(\theta):=\rho^2-|\theta|^2$. Since $\partial_j\lambda_\rho=\partial_j\lambda$ the proof of Lemma \ref{l:order-increasing} gives
\begin{equation}\label{e:dist-id2}
	\langle\lambda_\rho^\ast \chi_+^{-k}, \phi \rangle = \langle \lambda_\rho^\ast\chi_+^{-1}, L^{k-1}\phi \rangle.
\end{equation}
We denote by $T_{k, \rho}$ the Bochner--Riesz operator of which multiplier is the pullback of $\chi_+^{-k}$ via $\lambda_\rho$ instead of $\lambda$, that is, 
\[	T_{k,\rho}f=\mathcal F^{-1}\big( (\lambda_\rho^\ast\chi_+^{-k}) \mathcal F f \big).	\]
The estimates \eqref{e:st-BR}--\eqref{e:endpt-BR} also hold for $T_{k,\rho}$ with bounds $O(\rho^{\frac np-\frac nq-2k})$. 

When a test function $\phi(\theta, \tau, y)$ depends on several parameters including the variable $\theta$, we write $\langle T, \phi\rangle = \langle T, \phi\rangle_\theta$ to clarify that a distribution $T$ acts on the function $\theta\mapsto \phi(\theta, \tau, y)$.
\begin{lemma}\label{lem_BR}
Let $d$ and $k$ be positive integers such that $1\le k < d/2$ and let $\phi\in C_0^\infty([-2,2])$. If $(\frac1p,\frac1q)\in \mathcal P_k^d \cup [B_{k}^d, D_{k}^d]$, then 
\begin{equation} \label{e:cBR}
	\bigg\| \int_\R \phi(\tau) e^{it\tau}\big\langle \lambda_\rho^* \chi_+^{-k}, \widehat f(\theta,\tau) e^{iy\cdot \theta}\big\rangle_\theta d\tau\bigg\|_{L^{q,\infty}_x(\R^d)}\lesssim \|\phi\|_{C^2} \|f\|_{L^{p,1}(\R^d)}
\end{equation}
uniformly in $\rho\sim 1$. Here $x=(y,t)$ and $(\theta, \tau)$ denote the spatial and frequency variables, respectively, in $\R^{d-1}\times \R$. In fact, if  $(\frac1{p},\frac1{q})\in (B_{k}^d, D_{k}^d]$, then the $L^{q,\infty}$ in \eqref{e:cBR} can be replaced with $L^q$. Furthermore, if $(\frac1{p},\frac1{q})\in \mathcal P_k^d$, then the stronger $L^p-L^q$ estimate holds. 
\end{lemma}

\begin{proof}
We only prove the restricted weak type estimate \eqref{e:cBR}. The stronger estimates for $(\frac1{p},\frac1{q})\in P_k^d \cup (B_{k}^d, D_{k}^d]$ follow similarly by using the estimates \eqref{e:st-BR} or \eqref{e:endpt-BR1} in place of \eqref{e:endpt-BR}. 

We note 
\[	\int \phi(\tau) e^{it\tau} \big\langle \lambda_\rho^* \chi_+^{-k}, \widehat f(\theta,\tau) e^{iy\cdot \theta} \big\rangle_\theta d\tau 
	= 2\pi \int  \phi^{\vee} (t-s) \big\langle \lambda_\rho^* \chi_+^{-k}, \mathcal F( f(\cdot,s))(\theta) e^{iy\cdot \theta} \big\rangle_\theta ds, 	\]
where $\mathcal F$ denotes the $(d-1)$-dimensional Fourier transform. Hence 
\begin{equation}\label{e:dual_BR}
	\int \phi(\tau) e^{it\tau} \big\langle \lambda_\rho^* \chi_+^{-k}, \widehat f(\theta,\tau) e^{iy\cdot \theta} \big\rangle_\theta d\tau= (2\pi)^d \int  \phi^{\vee} (t-s) T_{k, \rho} (f(\cdot, s))(y) ds, 
\end{equation}
where $T_{k, \rho}$ is the $d-1$ dimensional Bochner--Riesz operator of index $-k$.  We also use the following simple inequalities (see, e.g., \cite[p. 780]{jkl18})
\begin{equation}\label{e:easy}
 \|f\|_{L^{q,\infty}_x}\le \|f\|_{L^q_tL^{q,\infty}_y}\quad \text{and} \quad \|f\|_{L^p_tL^{p,1}_y}\le \|f\|_{L^{p,1}_x}.
\end{equation}

Making use of the equality \eqref{e:dual_BR}, the first inequality in \eqref{e:easy}, and Minkowski's inequality, 
\[	\bigg\| \int_\R \phi(\tau) e^{it\tau}\big\langle \lambda_\rho^* \chi_+^{-k}, \widehat f(\theta,\tau) e^{iy\cdot \theta}\big\rangle_\theta d\tau\bigg\|_{L^{q,\infty}_x} \lesssim \bigg\| \int_\R \phi^\vee(t-s) \| T_{k, \rho}(f(\cdot, s))\|_{L^{q,\infty}_y}ds \bigg\|_{L^{q}_t}.	\]
Applying Young's inequality, the estimate \eqref{e:endpt-BR} with $n=d-1$, and the second inequality in \eqref{e:easy}, we see that  
\[	\bigg\| \int_\R \phi^\vee(t-s) \| T_{k, \rho}(f(\cdot, s))\|_{L^{q,\infty}_y}ds \bigg\|_{L^{q}_t} 
\lesssim \| \phi^\vee\|_{L^r} \| f\|_{L^{p}_tL^{p,1}_y}\le  \| \phi^\vee\|_{L^r} \| f\|_{L^{p,1}_x}	\]
for $r\in [1,\infty]$ satisfying $1+\frac1{q}=\frac1{p} +\frac1r.$  Since $\supp\phi\subset[-2,2]$, we have $\|\phi^\vee\|_r\lesssim \|\phi\|_{C^2}$ for every $r\in [1,\infty]$. Therefore, we obtain \eqref{e:cBR}.  
\end{proof}

\section{Carleman inequalities} \label{s:c_wucp}
Let us start with three simple observations. First, if $v=0$ then \eqref{carl} is the Hardy--Littlewood--Sobolev inequality, which is valid if and only if  $1<p,q<\infty$ satisfy the gap condition
\begin{equation}\label{e:gap}
\tfrac1p-\tfrac1q=\tfrac{2k}d.
\end{equation} 
Hence it is enough to consider $v\neq 0$. Secondly, since the estimate \eqref{carl} is invariant under rotations we may assume $v=|v|e_d$. Thirdly, by scaling it is easy to see that the gap condition \eqref{e:gap} is necessary for \eqref{carl} to hold.

Under the condition \eqref{e:gap}, by rescaling the Carleman inequality \eqref{carl} (with nonzero $v$) is equivalent to the following estimate for a single Fourier multiplier operator:
\begin{equation}\label{main}
\bigg \|\mathcal F^{-1} \bigg( \frac{ \widehat{f}(\xi) }{(|\xi|^2+2i \xi_d -1)^k}\bigg)\bigg\|_{L^q(\R^d)} 
 \le C \| f\|_{L^p(\R^d)}.
 \end{equation}

We first make a decomposition in $\xi_d$. Let $\psi\in C_0^\infty([-2,-1/2]\cup[1/2,2])$ be a nonnegative even function such that $\sum_{j\in\Z} \psi(2^{-j}t) =1$ for $t\neq0$, and let $\psi_0(t)=1-\sum_{j\ge 1}\psi(2^{-j}t)$. We write $\xi =(\eta,\tau)\in \mathbb R^{d-1}\times \mathbb R$ and, fixing a small positive dyadic number $\ep_\circ\le 2^{-5}$, set $\mathbb D:=\{2^\nu \in (0,\ep_\circ] \colon \nu \in \mathbb Z \}$. Then, we break 
\begin{equation}\label{e:prim_decom}
	\frac{1}{(|\eta|^2-1+\tau^2+2i \tau)^k} = m_{L}(\eta,\tau) + m_{G}(\eta, \tau)
\end{equation}
where
\begin{equation} \label{e:mtau}
	 m_L(\eta,\tau) = \sum_{\ep\in\mathbb D}m_\ep (\eta, \tau) := \sum_{\ep\in\mathbb D} \frac{\psi_0(\ep_\circ^{-1}(1-|\eta|^2))\psi(\ep^{-1}\tau)}{(|\eta|^2-1+\tau^2 + 2i \tau)^k} .
\end{equation}
Since $m_G$ vanishes in a neighborhood of the set $S^{d-2}\times\{0\}$ where the original multiplier is singular, it is clear that 
\[	|\partial^\alpha_\xi \big (|\xi|^{2k} m_G(\xi) \big)|\lesssim |\xi|^{-|\alpha|}	\] 
for every $\alpha\in \mathbb N_0^d$.  Thus, the Mikhlin multiplier theorem and the Hardy--Littlewood--Sobolev inequality yield the inequalities
\[	\|m_G(D)f\|_{L^q(\mathbb R^d)} \lesssim \|\mathcal F^{-1}(|\cdot|^{-2k}\widehat f\,)\|_{L^q(\mathbb R^d)} \lesssim \|f\|_{L^p(\mathbb R^d)}	\]
for all $p,q$ satisfying $1< p\le q< \infty$ and \eqref{e:gap}. Thus, the remaining task is to clarify $L^p-L^q$ boundedness of $m_L(D)$.

\subsection{Estimates for localized frequencies}
In the summation \eqref{e:mtau}, if we rescale each term $m_\ep$ by $\tau\to \ep\tau$, then 
\begin{equation}\label{e:scale}
	\|m_\ep(D)\|_{p\to q}= \ep^{\frac1p-\frac1q}\|\wt m_\ep(D) \|_{p\to q},
\end{equation} 
where
\begin{equation}\label{e:msc}
	\wt{m}_\ep(\eta, \tau)= m_\ep(\eta, \ep\tau)= \frac{ \psi_0(\ep_\circ^{-1}(1-|\eta|^2)) \psi(\tau)}{(|\eta|^2-1+ \ep^2\tau^2 + 2i \ep\tau)^k} .
\end{equation}
Hence, it is enough to study $L^p-L^q$ boundedness of $\wt m_\ep(D)$ instead of $m_\ep(D)$, which is the main result in this section (see Proposition \ref{p:tilde}).

Let us set $C= (\frac12, 0 )$ and  
\[	A^d = \big(\tfrac12, \tfrac{d-2}{2d} \big) \ \ \text{for} \ \  d\ge3.	\]
For a positive integer $k$ such that $k<d/2$, we define
\begin{align*}
\mathcal T^d_k
&=[A^d, B^d_k, C, D^d_k]\setminus[B^d_k, D^d_k] \\
&=\big\{ (x,y)\in \mathbf{Q}\colon \tfrac12\le x<\tfrac{d-2+2k}{2(d-1)}, \, 0\le y\le \tfrac{d-2}d(1-x) \big\}.
\end{align*}
See Figures \ref{fig1} and \ref{fig2}.

\begin{prop}\label{p:tilde}
Let $\ep\in \mathbb D$. If $(\frac1p,\frac1q)\in\mathcal T^d_k,$ then 
\begin{equation}\label{e:tilde}
    \|\wt m_\ep(D)\|_{p\to q} \lesssim \ep^{\frac{d-1}p-\frac{d-2+2k}2}.
\end{equation}
\end{prop}

The estimate \eqref{e:tilde}, combined with the identity \eqref{e:scale}, yields the following. 

\begin{corollary}\label{p:suff}
Let $\ep\in\mathbb D$. If $(\frac1p,\frac1q)\in\mathcal T_{k}^d$, then 
\begin{equation}\label{e:me}
\| m_{\ep} (D)\|_{p\to q} \lesssim \ep^{\frac{d}p-\frac1q -\frac{d-2+2k}2}.
\end{equation}
In particular, if $(\frac1p,\frac1q)\in [A^d,B_{k}^d)$, then
\begin{equation}\label{e:me1}
	\|m_\ep(D)\|_{p\to q} \lesssim \ep^{\frac{d+2}2(\frac1p-\frac1q)-k}.
\end{equation}
\end{corollary}

The second estimate \eqref{e:me1} follows by \eqref{e:me} since $[A^d, B_{k}^d)\subset \big\{(x,y) \colon y=\tfrac{d-2}d(1-x)\big\}$. In fact, all the estimates \eqref{e:tilde}, \eqref{e:me}, and \eqref{e:me1} are sharp, which we prove in the next section (Propositions \ref{p:example1} and \ref{p:example2}). 

In order to prove Proposition \ref{p:tilde} we use induction on $k$. For this, it is convenient to slightly generalize the definition of $\wt m_\ep$ as follows. For  $\zeta\in C^\infty_0([-2,2])$ and $0<\delta<1/2$, let us define
\begin{equation}
    \wt m^k_\ep[\zeta,\delta](\eta, \tau)
    =\frac{\zeta(\delta^{-1}(1-|\eta|^2))\psi(\tau)}{(|\eta|^2-1+\ep^2\tau^2 +2i\ep\tau)^k},\ \  (\eta,\tau)\in\R^{d-1}\times \R.
\end{equation}
When $\delta=\ep_\circ$, we simply write $\wt m^k_\ep[\zeta]=\wt m^k_\ep[\zeta,\ep_\circ]$.  Also, notice that the multiplier $\wt m_\ep$ defined in \eqref{e:msc} can be expressed as $\wt m_\ep=\wt m_\ep^k [\psi_0]= \wt m_\ep^k [\psi_0,\ep_\circ]$. In the next lemma, we obtain identities that are important for our induction arguments in the proof of Proposition \ref{p:suff}.
\begin{lemma}\label{l:counter}
Let $k$ be a positive integer and $L$ the differential operator defined in Lemma \ref{l:order-increasing} with $n$ replaced by $d-1$. That is,  
\[ Lh(\eta)=\frac1{2|\eta|^2}(d-3+\eta\cdot\nabla)h(\eta), \ \  \eta \in \R^{d-1}. \]
For $h\in C^{\infty}(\R^{d-1})$ supported away from the origin, we have
\begin{align}
\label{e:induc}
   & \langle \wt m^k_\ep[\zeta](\cdot, \tau), h\rangle 
    =\sum_{\ell =0}^{k-1} \frac{(-1)^\ell}{\ep^\ell_\circ (k-1-\ell)! \ell!} \langle \wt m^1_\ep [\zeta^{(\ell)}](\cdot, \tau), L^{k-1-\ell}h\rangle, \\
\label{e:rev}
	&\langle \wt m^1_\ep[\zeta,\delta](\cdot, \tau) , L^{k-1}h\rangle 
=  \sum_{\ell=0}^{k-1} \frac{(k-1)!}{\delta^\ell \ell!} \langle \wt m^{k-\ell}_\ep[\zeta^{(\ell)},\delta](\cdot, \tau), h\rangle.
\end{align}
\end{lemma}
\begin{proof}
First, we prove the identity \eqref{e:induc}. It is trivial if $k=1$, so we assume $k\ge2$. Direct differentiation gives 
\[ \wt m^k_\ep[\zeta](\eta, \tau) = -\frac1{k-1} \bigg(\frac1{\ep_\circ} \wt m^{k-1}_\ep [\zeta'] (\eta, \tau) +\frac\eta{2|\eta|^2} \cdot \nabla_\eta \wt m_\ep^{k-1}[\zeta] (\eta, \tau) \bigg).
\]
Hence, integrating by parts we have
\begin{align*}
\langle \wt m^k_\ep[\zeta] (\cdot, \tau), h\rangle
&=-\frac1{k-1} \bigg(\frac1{\ep_\circ} \langle \wt m^{k-1}_\ep [\zeta'](\cdot, \tau) , h\rangle - \sum_{j=1}^{d-1} \langle \wt m_\ep^{k-1}[\zeta](\cdot, \tau) , \frac \partial{\partial\eta_j}(\frac{\eta_j}{2|\eta|^2} h)\rangle \bigg)\\
&=-\frac1{(k-1)\ep_\circ} \langle \wt m^{k-1}_\ep [\zeta'](\cdot, \tau) , h\rangle + \frac1{k-1} \langle \wt m_\ep^{k-1}[\zeta](\cdot, \tau), L h\rangle.
\end{align*}
By this identity and the induction hypothesis we obtain 
\begin{align*}
\langle \wt m^k_\ep[\zeta] (\cdot, \tau), h\rangle
&= -\frac1{(k-1)\ep_\circ}\sum_{\ell =0}^{k-2} \frac{(-1)^\ell}{\ep^\ell_\circ (k-2-\ell)!\ell!} \langle \wt m^1_\ep [\zeta^{(\ell+1)}](\cdot, \tau), L^{k-2-\ell} h\rangle \\
&  \quad  + \frac1{k-1} \sum_{\ell =0}^{k-2} \frac{(-1)^\ell}{\ep^\ell_\circ (k-2-\ell)! \ell!} \langle \wt m^1_\ep [\zeta^{(\ell)}](\cdot, \tau), L^{k-1-\ell} h\rangle.
\end{align*}
Rearranging the summands gives \eqref{e:induc}.

Since $L^t=-\frac{\eta}{2|\eta|^2}\cdot \nabla$ we see that 
\begin{align}
\langle \wt m^j_\ep[\zeta,\delta](\cdot, \tau), L h \rangle 
&= \langle L^t \wt m^j_\ep[\zeta,\delta](\cdot, \tau), h \rangle \nonumber \\
&= \frac1{\delta} \langle \wt m^j_\ep[\zeta',\delta](\cdot, \tau), h \rangle +j \langle \wt m^{j+1}_\ep[\zeta,\delta](\cdot, \tau), h \rangle \label{e:iind}
\end{align}
for $j\in \mathbb N$. Now, we prove the second identity \eqref{e:rev}.  It is clear when $k=1$, so let us assume $k\ge2$. By the induction hypothesis,
\[\langle \wt m^1_\ep[\zeta,\delta](\cdot, \tau) , L^{k-1}h\rangle
= \sum_{\ell=0}^{k-2} \frac{(k-2)!}{\delta^\ell \ell!} \langle \wt m^{k-1-\ell}_\ep[\zeta^{(\ell)},\delta](\cdot, \tau), Lh\rangle.	\]
By \eqref{e:iind} this is equal to 
\[	\sum_{\ell=0}^{k-2} \frac{(k-2)!}{\delta^\ell \ell!} \bigg(\frac1\delta \langle \wt m^{k-1-\ell}_\ep[\zeta^{(\ell+1)},\delta](\cdot, \tau), h\rangle + (k-1-\ell)\langle \wt m^{k-\ell}_\ep[\zeta^{(\ell)},\delta](\cdot, \tau), h\rangle \bigg),	\]
from which \eqref{e:rev} follows.
\end{proof}

In the following, making use of Tomas--Stein's restriction estimate (\cite{To75, St-beijing}), we prove sharp $L^2-L^{\frac{2d}{d-2}}$ estimate for the multiplier operators given by $\wt m_\ep^k[\zeta, 2^j\ep]$.
\begin{lemma}\label{l:L2}
Let $1\le k <\frac d2$ and $\frac{2d}{d-2}\le q\le\infty$. For $\zeta \in C^\infty_0 \big([-2,2]\setminus [-\frac12,\frac12] \big)$, $\ep\in\mathbb D$, and $j=0,1,\ldots$ satisfying $2^j\le \frac1{4\ep}$, we have
\[    \|\wt m^k_\ep[\zeta, 2^j\ep](D)\|_{2\to q}\lesssim (2^j\ep)^{\frac12-k} \|\zeta\|_{L^\infty(\R)}.	\]
When $j=0$ the estimate also holds with $\zeta \in C^\infty_0([-2,2])$.
\end{lemma}

\begin{proof}
For brevity's sake, let us set 
\[	M^k_{\ep,j} = \wt m^k_\ep[\zeta, 2^j\ep].	\]
By interpolation it is enough to prove that
\begin{align}
	\| M^k_{\ep, j}(D)f\|_{L^\infty(\R^d)} &\lesssim (2^j\ep)^{\frac12-k} \|\zeta\|_{L^\infty(\R)} \|f\|_{L^2(\R^d)}, \label{es_L2} \\
	\| M^k_{\ep, j}(D)f\|_{L^\frac{2d}{d-2}(\R^d)} &\lesssim (2^j\ep)^{\frac12-k} \|\zeta\|_{L^\infty(\R)} \|f\|_{L^2(\R^d)}. \label{es_TS}
\end{align}

Obviously,
\begin{align*}
	|M^k_{\ep, j}(D)f(x)| 
	&\lesssim \|M^k_{\ep, j}\|_{L^2(\mathbb R^d)} \| \wh f\|_{L^2(\mathbb R^d)} 
	\lesssim |\supp M^k_{\ep,j}|^\frac12 \|M^k_{\ep, j}\|_{L^\infty(\mathbb R^d)} \|f\|_{L^2(\mathbb R^d)}.
\end{align*}
If $\zeta\in C^\infty_0 \big([-2,2]\setminus [-\frac12,\frac12] \big)$, then  $M^k_{\ep,j}$ is supported in 
\begin{equation}\label{e:supp}
\begin{aligned}
	&\big\{(\eta, \tau)\in \R^{d-1}\times \R\colon\sqrt{1-2^{j+1}\ep} \le |\eta| \le \sqrt{1-2^{j-1}\ep}, \ \tfrac12\le |\tau|\le 2\big\} \\
	&\cup 	\big\{(\eta, \tau)\in \R^{d-1}\times \R\colon\sqrt{1+2^{j-1}\ep} \le |\eta| \le \sqrt{1+2^{j+1}\ep}, \ \tfrac12\le |\tau|\le 2\big\}.
\end{aligned}
\end{equation}
Hence it is clear that $|\supp M^k_{\ep,j}|\lesssim 2^j\ep$.  On the support of $M^k_{\ep,j}$ we have
\[	||\eta|^2-1+\ep^2\tau^2 +2i\ep\tau| \ge ||\eta|^2-1|-\ep^2\tau^2 \ge 2^{j-1}\ep-4\ep_\circ\ep,	\]
which is greater than $2^{j-2}\ep$ since $4\ep_\circ \le 2^{-3}< 2^{j-2}$. Hence
\begin{equation}\label{e:M-sup}
\|M^k_{\ep, j}\|_{L^\infty(\mathbb R^d)} \lesssim (2^{j}\ep)^{-k} \|\zeta\|_{L^\infty(\R)},
\end{equation} and we obtain the estimate \eqref{es_L2}.

If $\supp \zeta \subset [-\frac12, \frac12]$, then $\supp M_{\ep, 0}^k$ is contained in the set
\begin{equation}\label{e:supp0}
	\big\{(\eta, \tau)\in \R^{d-1}\times \R\colon\sqrt{1-2\ep} \le |\eta| \le \sqrt{1+2\ep}, \ \tfrac12\le |\tau|\le 2 \big\}
\end{equation}
on which $||\eta|^2-1+\ep^2\tau^2 +2i\ep\tau| \ge \ep$. Thus, a similar argument gives \eqref{es_L2} with $j=0$ and $\zeta\in C_0^\infty([-2,2])$.

To prove \eqref{es_TS}, let us choose $\wt\psi \in C^\infty_0([-3,3])$ such that $\wt\psi \zeta =\zeta$. Writing $x=(y,t)\in \R^{d-1}\times \R$ and using the spherical coordinates we note that  
\[	M^k_{\ep, j}(D)f(x)=\frac1{(2\pi)^d} \int_0^\infty \wt\psi \bigg(\frac{1-\rho^2}{2^j\ep}\bigg)  \int_\R\int_{S^{d-2}} e^{i(\rho y\cdot\theta+t\tau)} \big(M^k_{\ep,j}\wh f \big)(\rho\theta, \tau) d\sigma(\theta) d\tau \rho^{d-2}d\rho .	\] 
Thus, by the Minkowski inequality $\|M^k_{\ep,j}(D)f\|_{L^\frac{2d}{d-2}(\R^d)}$ is bounded by a constant times 
\begin{equation}\label{e:st_me}
	\int_0^\infty  \wt\psi \bigg(\frac{1-\rho^2}{2^j\ep}\bigg) \left\| \int_\R\int_{S^{d-2}} e^{i(\rho y\cdot\theta+t\tau)} \big(M^k_{\ep,j}\wh f \big)(\rho\theta, \tau) d\sigma(\theta) d\tau \right\|_{L_{y,t}^\frac{2d}{d-2}} \rho^{d-2}d\rho.
\end{equation}
Applying the Hausdorff--Young and Minkowski inequalities successively gives 
\begin{align*}
	&\left\| \int_\R\int_{S^{d-2}} e^{i(\rho y\cdot\theta+t\tau)} \big(M^k_{\ep,j}\wh f \big)(\rho\theta, \tau) d\sigma(\theta) d\tau \right\|_{L_{y,t}^{\frac{2d}{d-2}}(\R^d)} \\
	&\lesssim \bigg\| \left\| \int_{S^{d-2}} e^{i\rho y\cdot\theta} \big(M^k_{\ep,j}\wh f \big)(\rho\theta, \tau) d\sigma(\theta) \right\|_{L_\tau^\frac{2d}{d+2}(\R)} \bigg\|_{L_{y}^{\frac{2d}{d-2}}(\R^{d-1})} \\
	&\le \bigg\| \left\| \int_{S^{d-2}} e^{i\rho y\cdot\theta} \big(M^k_{\ep,j}\wh f \big)(\rho\theta, \tau) d\sigma(\theta) \right\|_{L_{y}^{\frac{2d}{d-2}}(\R^{d-1})} \bigg\|_{L_\tau^\frac{2d}{d+2}(\R)} .
\end{align*}
Since $2^j\ep \le \frac14$ we see from \eqref{e:supp} (or \eqref{e:supp0}) that $M^k_{\ep,j}(\rho\theta,\tau)\neq 0$ only if $\rho\sim 1$. Making use of the adjoint restriction estimate due to Tomas \cite{To75} and Stein \cite{St-beijing} and the H\"older inequality, the last is 
\[	 \lesssim \big\| \big\|\big(M^k_{\ep,j}\wh f \big)(\rho\theta, \tau)\big\|_{L^2_\theta(S^{d-2})} \big\|_{L_\tau^\frac{2d}{d+2}(\R)} \lesssim \big\|\big(M^k_{\ep,j}\wh f \big)(\rho\theta, \tau)\big\|_{L^2_{\theta, \tau}(S^{d-2}\times \R)}. 	\]
Since the $\rho$-support of the integrand in \eqref{e:st_me} is contained in an interval of length $\lesssim 2^j\ep$, the Cauchy--Schwarz inequality, \eqref{e:M-sup}, and the Plancherel theorem yield
\begin{align*}
\|M^k_{\ep, j}(D)f\|_{L_x^\frac{2d}{d-2}} 
	&\lesssim \sqrt{2^j\ep} \bigg(\int_0^\infty  \wt\psi \bigg(\frac{1-\rho^2}{2^j\ep}\bigg)^2 \big\| ( M^k_{\ep,j}\wh f \big)(\rho\theta, \tau)\big\|_{L^2_{\theta, \tau}(S^{d-2}\times \R)}^2 \rho^{d-2}d\rho \bigg)^\frac12 \\
	&= \sqrt{2^j\ep} \, \big\| \big( M^k_{\ep,j} \wh f \big)(\eta, \tau)\big\|_{L^2_{\eta, \tau}(\mathbb R^d)} \\
	&\lesssim (2^j\ep)^{\frac12-k} \|\zeta\|_{L^\infty(\R)} \|f\|_{L^2(\mathbb R^d)}.
\end{align*}
This gives the estimate \eqref{es_TS}.
\end{proof}

Now, we prove Proposition \ref{p:tilde}.
\begin{proof}[Proof of Proposition \ref{p:tilde}]  
Let $h_{y,\tau}(\eta)=\wh f(\eta,\tau)e^{i\eta\cdot y}$. We recall the notation $\wt m_\ep= \wt m_\ep^k[\psi_0]$, and write  
\[	\wt m_\ep(D) f(y,t) = \frac1{(2\pi)^d}\int_\R e^{it\tau} \big \langle \wt m_\ep^k[\psi_0](\cdot, \tau) , h_{y,\tau} \big \rangle d\tau. \]
As  mentioned before, we need to dyadically decompose the inner integral near $S^{d-2}$, and obtain sharp estimates for each dyadic piece. More precisely, for $(\frac1p, \frac1q)\in [B_k^d, D_k^d]$, we show the $\ep$-uniform $L^{p,1}-L^{q,\infty}$ estimates for the dyadic operators. When $k\ge2$, however, this is not possible as can be seen by a simple example.  Prior to dyadic decomposition, we must relax the order $k$ by using the identity \eqref{e:induc}. This gives, for $\ep\in \mathbb D$,
\[	\wt m_\ep(D) f(y,t) = \sum_{\ell=0}^{k-1} c_{k,\ell,\ep_\circ} \int_\R e^{it\tau} \big \langle \wt m^1_\ep [\psi_0^{(\ell)}](\cdot, \tau), L^{k-1-\ell} h_{y,\tau} \big\rangle d\tau, \]
where $c_{k,\ell,\ep_\circ}=\frac{(-1)^\ell}{(2\pi)^d \ep_\circ^\ell (k-1-\ell)! \ell!}$. 

Let us define 
\[	I_{\ell}(y,t)= \int_\R e^{it\tau} \big\langle \wt m^1_\ep [\psi_0^{(\ell)}](\cdot, \tau), L^{k-1-\ell} h_{y,\tau} \big\rangle d\tau.	\]
For $0\le \ell\le k-1$ and $(\frac1p,\frac1q)\in \mathcal T_k^d$, we need only to prove 
\begin{equation}\label{e:goal}
    \| I_\ell \|_{L^q(\R^d)} \lesssim \ep^{\frac{d-1}p-\frac{d-2+2k}2}\|f\|_{L^p(\R^d)}.
\end{equation}

We dyadically decompose $\wt m^1_\ep [\psi_0^{(\ell)}](\cdot,\tau)$ away from $S^{d-2}$ in the $\ep$ scale. Recalling the smooth cutoff function $\psi$ introduced at the beginning of Section \ref{s:c_wucp}, we define $\psi_j(t)=\psi(2^{-j}t)$ for $j\ge1$. Then $\sum_{j\ge 0} \psi_j =1$ on $\R$, so we can write $I_\ell(y,t) = \sum_{j\ge 0} I_{\ell, j}(y,t)$ where
\begin{equation}\label{e:def_I}
I_{\ell, j}(y,t):= \int_\R e^{it\tau} \big\langle  \psi_j (\ep^{-1}(1-|\cdot|^2))\wt m^1_\ep [\psi_0^{(\ell)}](\cdot, \tau), L^{k-1-\ell} h_{y,\tau} \big\rangle d\tau.
\end{equation}
In fact, the set of summation indices $j$ is finite since $\psi_j(\cdot/\ep)\psi_0^{(\ell)}(\cdot/\ep_\circ)\neq 0$ only if $2^{j-1}\ep\le 2\ep_\circ$.  

We aim to prove that 
\begin{gather}
\| I_{\ell,j}\|_{L^{q,\infty}(\R^d)} \lesssim \|f\|_{L^{p,1}(\R^d)} \ \  \text{if} \ \  (\tfrac1p,\tfrac1q)\in [B^d_k, D^d_k], \label{e_lj}\\
\| I_{\ell,j}\|_{L^q(\R^d)} \lesssim (2^j\ep)^{\frac12-k}\|f\|_{L^2(\R^d)} \ \  \text{if} \ \  \tfrac{2d}{d-2}\le q\le \infty. \label{e_lj2}
\end{gather}
If we assume these estimates for the moment, then we have by real interpolation
\[	\| I_{\ell,j}\|_{L^q(\R^d)} \lesssim (2^j\ep)^{\frac{d-1}p-\frac{d-2+2k}2}\|f\|_{L^p(\R^d)}	\]
for $(\frac1p,\frac1q)\in \mathcal T^d_k$. Since $\frac1p<\frac{d-2+2k}{2(d-1)}$ summing over $j$ we obtain the desired estimate \eqref{e:goal}.  It remains to prove \eqref{e_lj} and \eqref{e_lj2}.

First, we prove \eqref{e_lj}. Using the spherical coordinates, the identity \eqref{e:sp_msr}, and scaling,  we have 
\[	I_{\ell,j}(y,t)=\int_0^\infty \psi_j\bigg(\frac{1-\rho^2}{\ep}\bigg) \int_\R e^{it\tau}\phi_{\ep, \ell}(\rho, \tau) \big\langle \lambda_\rho^*\chi_+^{-1}, L^{k-1-\ell} h_{y,\tau}\big\rangle d\tau d\rho,	\]
where $\lambda_\rho(\theta):=\rho^2-|\theta|^2$ and
\[	\phi_{\ep,\ell}(\rho, \tau):= \frac{\psi_0^{(\ell)}(\ep_\circ^{-1}(1-\rho^2))\psi(\tau)}{\rho^2-1+\ep^2\tau^2+2i\ep\tau}.  	\]
Here we omitted  harmless constant multiplication depending only on $d$.  The identity \eqref{e:dist-id2} gives
\begin{equation}\label{e:BR}
	I_{\ell, j}(y,t) = \int \psi_j\bigg(\frac{1-\rho^2}{\ep}\bigg) \int e^{it\tau}\phi_{\ep, \ell}(\rho, \tau) \big\langle \lambda_\rho^*\chi_+^{-(k-\ell)}, \wh f(\theta,\tau)e^{i\theta\cdot y} \big\rangle_\theta d\tau d\rho.
\end{equation}
Let us denote by $\mathcal I_j$ the support of the function $\rho\mapsto \psi_j((1-\rho^2)/\ep)$, that is, 
\[ \mathcal I_j	:=\begin{cases}
	\big [\sqrt{1-2^{j+1}\ep}\,,\sqrt{1-2^{j-1}\ep} \,\big] \cup \big[ \sqrt{1+2^{j-1}\ep}\, ,\sqrt{1+2^{j+1}\ep} \,\big]  &\text{if} \ \ j\ge 1, \\[3pt]
	\big [\sqrt{1-2\ep}\,,\sqrt{1+ 2\ep} \,\big]  &\text{if} \ \ j=0.
	\end{cases}	\]
From the Minkowski inequality and Lemma \ref{lem_BR} it follows that 
\begin{align*}
&\| I_{\ell,j}\|_{L^{q,\infty}(\R^d)}\\
    &\le \int_0^\infty \psi_j\bigg(\frac{1-\rho^2}{\ep}\bigg) \bigg\| \int_\R e^{it\tau}\phi_{\ep, \ell}(\rho, \tau) \big\langle \lambda_\rho^*\chi_+^{-(k-\ell)}, \wh f(\theta,\tau)e^{i\theta\cdot y} \big\rangle_\theta d\tau \bigg\|_{L_{y,t}^{q,\infty}(\R^d)} d\rho \\
    &\lesssim 2^j\ep \sup_{\rho\in\mathcal I_j} \bigg\| \int_\R e^{it\tau}\phi_{\ep, \ell}(\rho, \tau) \big\langle \lambda_\rho^*\chi_+^{-(k-\ell)}, \wh f(\theta,\tau)e^{i\theta\cdot y} \big\rangle_\theta d\tau \bigg\|_{L_{y,t}^{q,\infty}(\R^d)} \\
    &\lesssim 2^j\ep \bigg[ \sup_{\rho\in \mathcal I_j}\| \phi_{\ep,\ell}(\rho,\cdot) \|_{C^2} \bigg] \|f\|_{L^{p,1}(\R^d)} 
\end{align*}
for $(\frac1p,\frac1q)\in[B^d_k, D^d_k]$. Elementary calculation shows that $\sup_{\rho\in \mathcal I_j} \|\phi_{\ep,\ell}(\rho,\cdot)\|_{C^2}\lesssim (2^j\ep)^{-1}$ (see Lemma \ref{l:c2-norm} below). Hence we obtain \eqref{e_lj}.

Now, we turn to prove \eqref{e_lj2}. We cannot follow the strategy of the proof of \eqref{e_lj} which relies on \eqref{e:BR} and boundedness of the Bochner--Riesz operator $T_{k-\ell}$ of order $-(k-\ell)$, since $T_{k-\ell}$ is unbounded from $L^2$ to $L^q$ for any $1\le q\le \infty$ and $0\le \ell \le k-1$. To get over the issue, using the identity \eqref{e:rev}, we integrate by parts again (in the definition \eqref{e:def_I} of $I_{\ell,j}$) to remove $L^{k-1-\ell}$ and then apply Lemma \ref{l:L2}. 

Let us define $\zeta_0\in C_0^\infty([-2,2])$ and $\zeta_j\in C_0^\infty([-2,2]\setminus [-1/2,1/2])$, $j\ge1$, by setting
\[	\zeta_j(t)= \begin{cases}
\psi_0^{(\ell)} (\ep t/\ep_\circ) \psi_0(t) &\text{if} \ \ j=0, \\
\psi_0^{(\ell)} (2^j\ep t/\ep_\circ) \psi(t) &\text{if} \ \ j\ge1.
\end{cases}	\]
Then 
\[ \psi_j (\ep^{-1}(1-|\eta|^2))\wt m^1_\ep [\psi_0^{(\ell)}](\eta, \tau) = \wt m^1_\ep[\zeta_j,2^j\ep ](\eta,\tau), \ \ j\ge 0,  \]
so 
\[	I_{\ell,j}(y,t) =  \int_\R e^{it\tau} \big\langle \wt m^{1}_\ep[\zeta_j, 2^j\ep](\cdot,\tau), L^{k-1-\ell}h_{y,\tau} \big\rangle d\tau.	\]
Applying the identity \eqref{e:rev} we have  
\begin{align*}
 I_{\ell,j}(y,t) 
 &= \sum_{r=0}^{k-1-\ell} \frac{(k-1-\ell)!}{(2^j\ep)^r r!} \int_\R e^{it\tau} \big\langle \wt m^{k-\ell-r}_\ep[\zeta_j^{(r)}, 2^j\ep](\cdot,\tau), h_{y,\tau} \big\rangle d\tau\\
 & = \sum_{r=0}^{ k-1-\ell} \frac{(k-1-\ell)!}{(2^j\ep)^r r!} \iint \wt m^{k-\ell-r}_\ep[\zeta_j^{(r)}, 2^j\ep](\eta,\tau) \wh f(\eta,\tau)  e^{i(y\cdot \eta+t\tau)} d\eta d\tau.
 \end{align*}
We notice that $\supp\zeta_j^{(r)}\subset [-2,2]\setminus [-1/2,1/2]$ if $j\ge1$ and $\supp\zeta_0^{(r)}\subset [-2,2]$.  Since $2^j\ep\le 4\ep_\circ$, we also observe that $\|\zeta_j^{(r)}\|_{L^\infty(\R)}$ is bounded uniformly in $j$ and $\ep$. Thus, by Lemma \ref{l:L2} we get
\[	\|I_{\ell,j}\|_{L^q(\R^d)} 
	\lesssim \sum_{r=0}^{k-1-\ell} (2^j\ep)^{-r}  (2^j\ep)^{\frac12-(k-\ell-r)} \|f\|_{L^2(\R^d)} 
	\lesssim (2^j\ep)^{\frac12-k+\ell} \|f\|_{L^2(\R^d)}	\]
for $\frac{2d}{d-2}\le q\le \infty.$ Since $2^j\ep\lesssim 1$ this estimate gives \eqref{e_lj2}.
\end{proof}

\begin{lemma} \label{l:c2-norm}
Let $\mathcal I_j$ and $\phi_{\ep, \ell}$ be as in the Proof of Proposition \ref{p:tilde}. We have 
\[	\sup_{\rho \in \mathcal I_j} \|\phi_{\ep, \ell}(\rho, \cdot)\|_{C^2(\R)} \lesssim (2^j\ep)^{-1}. 	\]
\end{lemma}
\begin{proof}
If $\rho\in \mathcal I_j$, $j\ge1$, and $\tau\in \supp\psi$, then 
\[	|\rho^2-1+\ep^2\tau^2+2i\ep\tau|\ge |\rho^2-1|-|\ep^2\tau^2| \ge (2^{j-1}-4\ep_\circ)\ep\gtrsim 2^{j}\ep.	\]
The last inequality holds since $0<\ep\le \ep_\circ\le 2^{-5}$. On the other hand, if $\rho \in \mathcal I_0$ and $\tau\in \supp\psi$, then
\[	|\rho^2-1+\ep^2\tau^2+2i\ep\tau|\ge 2\ep\tau \ge \ep.	\]
Therefore, direct calculation shows, for $\rho\in \mathcal I_j$ and $r=0,1,2$, $|\partial_\tau^r \phi_{\ep, \ell}(\rho, \tau)|\lesssim (2^j\ep)^{-1}$.
\end{proof}

\subsection{Sharpness of the local estimates} \label{s:sharp}
The estimate \eqref{e:me} is sharp in the sense that the exponent of $\ep$ in \eqref{e:me} cannot be made larger. To show this, we begin with a general fact regarding $L^p-L^q$ norm of  Fourier multipliers. 
\begin{lemma}\label{l:adj_same}
Let $1\le p,q \le \infty$. For $a\in L^\infty(\R^d)$,
\begin{equation}\label{e:same}
	\|a(D)\|_{p\to q}=\|\overline{a} (D)\|_{p\to q}.
\end{equation}
Consequently, 
\begin{equation}\label{e:re_im}
	\|\RE a(D)\|_{p\to q} \le \|a(D)\|_{p\to q} \ \ \text{and} \ \ \|\IM a(D)\|_{p\to q} \le \|a(D)\|_{p\to q}.
\end{equation}
\end{lemma}
\begin{proof}
By the definition, $\overline{a}(D)f(x)= \overline{a(D)h(-x)}$ where $h$ is defined by $\wh h(\xi)= \overline{\wh f(\xi)}$ so that $h(x)=\overline{f(-x)}$. Hence
\[	\|\overline{a}(D)f\|_{L^q} = \|a(D)h\|_{L^q} \le \|a(D)\|_{p\to q} \|h\|_{L^p}= \|a(D)\|_{p\to q} \|f\|_{L^p},	\]
which shows $\|\overline{a} (D)\|_{p\to q}\le \|a(D)\|_{p\to q}$.  Similarly,  $\|a (D)\|_{p\to q}\le \| \overline{a}(D)\|_{p\to q}$, and this gives the identity \eqref{e:same}. The inequalities in \eqref{e:re_im} are clear. 
\end{proof}

Sharpness of the estimate \eqref{e:me1} can be proved by a Knapp type example adapted to the cylinder $S^{d-2}\times[1/2,2]\subset \R^d$.
\begin{prop}\label{p:example1}
Let $d$ and $k$ be positive integers and let $1\le p,q\le\infty$. Then, if $0<\ep\le\delc$ for some small $\delc>0$, 
\[
\|m_\ep(D)\|_{p\to q} \gtrsim \ep^{\frac{d+2}2(\frac1p-\frac1q)-k}.
\] 
\end{prop}
\begin{proof} By \eqref{e:scale} it is enough to prove 
\[	\|\wt m_\ep(D)\|_{p\to q}\gtrsim \ep^{\frac d2(\frac1p-\frac1q)-k}, 	\] 
and furthermore, by the second inequality in \eqref{e:re_im}, we need only show  
\begin{equation}\label{e:lower}
	\| \IM \wt m_\ep (D)\|_{p\to q}\gtrsim \ep^{\frac d2(\frac1p-\frac1q)-k}.
\end{equation}
We note that $\IM \wt m_\ep(\eta, \tau)$ is equal to 
\begin{equation}\label{e:imag}
 \frac{\psi_0(\ep_\circ^{-1}(1-|\eta|^2)) \psi(\tau)}{ ((|\eta|^2-1+\ep^2\tau^2)^2 + 4\ep^2 \tau^2 )^k} \sum_{l=1}^{\lfloor \frac{k+1}2\rfloor} (-1)^l\binom{k}{2l-1} (|\eta|^2-1+\ep^2\tau^2 )^{k-2l+1} (2\ep\tau)^{2l-1}.
\end{equation}

Let us choose a nonnegative smooth function $\phi$ on $\R$ such that $\supp \phi \subset [1/2,2]$ and $\phi=1$ on $[1,3/2]$, and define 
\[ \wh {f_\ep} (\eta, \tau) = \phi(\tau) \phi \bigg(\frac{\eta_{d-1}-1}{\delta_\circ \ep} \bigg)\prod_{j=1}^{d-2} \phi\bigg(\frac{\eta_j}{\sqrt{\delta_\circ \ep}}\bigg)	\]
for a small constant $\delta_\circ>0$ and $0<\ep\le \delc$. It is clear that 
\begin{equation}\label{e:fpn}
	\|f_\ep\|_{L^p} \sim (\delta_\circ \ep)^{\frac d2-\frac d{2p}}. 
\end{equation}

If $(\eta,\tau)\in\supp\wh {f_\ep}$, we have $\tau\sim 1$ and $|\eta|^2-1+\ep^2\tau^2 \sim \delta_\circ \ep	$. This yields  
\[	|\IM \wt m_\ep(\eta, \tau)| \gtrsim \frac1{\ep^{2k}} \bigg( \delta_\circ \ep^k - \sum_{l=2}^{\lfloor \frac{k+1}2\rfloor} \delta_\circ^{2l-1} \ep^k\bigg)  \sim \delta_\circ\ep^{-k} \]
if 
\[	(\eta,\tau)\in Q_\ep:= \big[\sqrt{\delta_\circ\ep}, \,\tfrac32\sqrt{\delta_\circ\ep}\, \big]^{d-2} \times \big[ 1+ \delta_\circ\ep, \, 1+\tfrac32\delta_\circ\ep \big] \times \big[1,\, \tfrac32\big]	\] 
whenever $\delta_\circ$ is small enough. Clearly, $\IM \wt m_\ep$ is either negative or positive on $Q_\ep$. Hence, if $x$ lies in the set
\[	S_\ep:=\big\{x\in \R^d\colon |x_d|\le 10^{-3}, \, |x_{d-1}|\le \ep^{-1}, \, |x_j|\le \ep^{-\frac12} \  \text{for} \  1\le j\le d-2 \big\}	\]
for $0<\ep\le \delc$ and $\delta_\circ$ is small enough, then 
\begin{align*}
	&|\IM \wt m_\ep (D)f_\ep (x)| \\
	&\ge \bigg| \iint \cos \Big( x_d\tau + x_{d-1}(\eta_{d-1}-1) + \sum_{j=1}^{d-2} x_j\eta_j \Big) \IM \wt m_\ep (\eta, \tau) \widehat{f_\ep}(\eta,\tau) d\eta d\tau \bigg| \\
	&\gtrsim |Q_\ep| \delta_\circ \ep^{-k} \sim (\delta_\circ\ep)^{\frac d2} \delta_\circ\ep^{-k}.
\end{align*}
Taking a sufficiently small $\delta_\circ$, we have
\[	\|\IM \wt m_\ep (D)f_\ep\|_{L^q(\R^d)} \ge \|\IM \wt m_\ep (D)f_\ep\|_{L^q(S_\ep)} \gtrsim \ep^{-\frac d{2q}} \ep^{\frac d2-k} 	\]
for $0<\ep\le \delta_\circ$. Combined with \eqref{e:fpn}, this implies
\[ \|\IM \wt m_\ep (D)\|_{p\to q} \ge \frac{\|\IM \wt m_\ep (D)f_\ep\|_{L^q(\R^d)}}{\|f_\ep\|_{L^p(\R^d)}} \gtrsim \ep^{\frac d2(\frac1p-\frac1q)-k},	\] 
which gives the  desired estimate \eqref{e:lower}.
\end{proof}

We can also prove that the estimate \eqref{e:me} is sharp. In particular, the estimate \eqref{e:me1} is sharp when $\frac1q =\frac{d-2}d(1-\frac1p)$.
\begin{prop}\label{p:example2}
Let $d$, $k$ be positive integers, $1\le p,q\le\infty$, and $0<\ep\ll 1$. Then we have
\begin{equation}\label{e:nec2}
\|m_\ep(D)\|_{p\to q} \gtrsim  \ep^{\frac{d}p-\frac1q -\frac{d-2+2k}2}.
\end{equation}
\end{prop}
By \eqref{e:scale} and duality, the estimate  \eqref{e:nec2} follows from
\begin{equation}\label{eq_goal}
\|\wt m_\ep(D)\|_{p\to q} \gtrsim  \ep^{-\frac{d-1}q + \frac{d}2-k}.
\end{equation}
When $k=1$, it is relatively simple to obtain the lower bound \eqref{eq_goal} by analyzing $\IM \wt m_\ep$ and using \eqref{e:re_im}. Indeed, this was done in \cite{jkl18}.  For larger $k$, however,  $\IM \wt m_\ep$ is given by a summation of $\sim k/2$ terms with \emph{alternating} signs (see \eqref{e:imag}). Furthermore, it can be shown that each of those yields a Fourier multiplier whose $L^p-L^q$ norm is $\gtrsim \ep^{-\frac{d-1}q +\frac d2-k}$. In other words, there is no `leading term' in the alternating sum \eqref{e:imag}. This makes it difficult to determine the lower bound for $\|\IM \wt m_\ep(D)\|_{p\to q}$. 

We get around this problem and prove \eqref{eq_goal} reducing the order of denominator of $\wt m_\ep$ by integration by parts. As the proof is rather involved, we shall postpone it until the last section; see Section \ref{s:pf_lower}.

\subsection{Proof of Theorem \ref{t:carl}}

As mentioned before, the condition \eqref{e:gap} is necessary for the Carleman inequality \eqref{carl}. In the preliminary decomposition \eqref{e:prim_decom}, the global part $m_G(D)$ is bounded from $L^p(\R^d)$ to $L^q(\R^d)$ if $p,q\in(1,\infty)$ satisfy  \eqref{e:gap}. The other condition in \eqref{e:cond} is determined by the local part $m_L(D)$.  

\begin{prop}\label{p:local}
Let $k$ be a positive integer such that $k< d/2$, and let $1< p, q<\infty$. Then
 \begin{equation}\label{e:local_est}
\|m_L(D)f\|_{L^q(\R^d)} \lesssim \|f\|_{L^p(\R^d)}
\end{equation}
if and only if
\begin{equation}\label{e:penta}
\tfrac 1p-\tfrac1q \ge \tfrac{2k}{d+2}, \ \ \tfrac dp- \tfrac 1q \ge \tfrac{d-2+2k}2, \ \ \text{and} \ \ \tfrac dq-\tfrac1p \le \tfrac{d-2k}2.
\end{equation}
\end{prop}

Before proving the proposition we additionally define two points $E_k^d$ and $F_k^d$ in $\mathbf{Q}$. For every $(d,k) \in \mathbb N\times \mathbb N$ satisfying $1\le k< d/2$, we define 
\[	E_{k}^d = \big( \tfrac{d^2+2kd-4}{2(d+2)(d-1)}, \tfrac{(d-2)(d+2-2k)}{2(d+2)(d-1)} \big), \ F_{k}^d =\big(\tfrac{d-2+2k}{2d}, 0 \big).	\]
We note that $E_k^d$ and $F_k^d$ are on the line $\{(x,y)\in\mathbf{Q} \colon dx-y=\frac{d-2+2k}2 \}$, while $E_k^d$ and $(E_k^d)'$ are on the line $\{(x,y) \colon x-y=\tfrac{2k}{d+2} \}$. See Figures \ref{fig1} and \ref{fig2}.  The pairs $(\frac1p, \frac1q)$ satisfying the conditions in \eqref{e:penta} are contained in the closed pentagon $[E_{k}^d, F_{k}^d, (E_{k}^d)', (F_{k}^d)', H]$.  

The line segments $[E_k^d, F_k^d]$ and $\mathbb L_k^d:=\{(x,y)\in \mathrm{int}\,\mathbf{Q} \colon x-y=\frac{2k}d\}$ intersect in the interior of $\mathbf{Q}$ if and only if $k<\frac{d-2}2$. In this case, we denote the intersection point by $G_{k}^d$, that is, 
\[	G_{k}^d=\big(\tfrac{(d+2k)(d-2)}{2d(d-1)}, \tfrac{d-2k-2}{2(d-1)} \big).	\]
Combining the conditions \eqref{e:gap} and \eqref{e:penta}, we conclude that the Carleman inequality \eqref{carl} holds if and only if $(\frac1p, \frac1q)$ lies on the line
\[  \mathbb L_k^d \cap [E_{k}^d, F_{k}^d, (E_{k}^d)', (F_{k}^d)', H], \]
which is $[G_k^d, (G_k^d)']$ if $k<\frac{d-2}2$ and $\mathbb L_k^d$ if $\frac{d-2}2\le k<\frac d2$. Therefore, the proof of Theorem \ref{t:carl} is completed.

\begin{proof}[Proof of Proposition \ref{p:local}]
 %%%%%%%%%%
\begin{figure}
\captionsetup{type=figure,font=footnotesize}
\centering
\begin{tikzpicture} [scale=0.6]\scriptsize
	\path [fill=lightgray] (0,0)--(15/2,0)--(15/2,3/2)--(0, 6)--(0,0);
	\draw (15/4, 2) node{$\mathcal T_{2}^5$};
	\draw [<->] (0,10.7)node[left]{$\frac1q$}--(0,0) node[below]{$\frac12$} node[left]{$0$} node[above right]{$C$}--(10.7,0) node[below]{$\frac1p$};
	\draw (0,10) node[left]{$\frac12$} --(10,10)--(10,0) node[below]{$1$} node[above right]{$H$};
	\draw [dash pattern={on 2pt off 1pt}] (0,6)node[above right]{$A^5$} node[left]{$\frac3{10}$}--(10,0); %the adjoint restriction line
	\draw [dash pattern={on 2pt off 1pt}] (4,10)node[above]{$(A^5)'$} --(10,0); %the restriction line
	\draw [dash pattern={on 2pt off 1pt}] (15/2,0)node[below]{$\frac78$} node[above right]{$D_{2}^5$}--(15/2,3/2); \node[above] at (15/2-0.1,3/2+0.05) {$B_{2}^5$}; %the line BD
	\draw [dash pattern={on 2pt off 1pt}](10-3/2,10-15/2)--(10,10-15/2)node[right]{$(D_{2}^5)'$}; %the line B'D'
	\draw [dash pattern={on 2pt off 1pt}](4,0) node[below]{$\frac7{10}$} node[above right]{$F_{2}^5$}--(65/14,45/14) --(10-45/14,10-65/14) --(10,6) node[right]{$(F_{2}^5)'$}; %the line FEE'F'
	\node [left] at (10-3/2+0.03,10-15/2+0.05) {$(B_{2}^5)'$};
	\node[left] at (10-45/14,10-65/14) {$(E_{2}^5)'$};
	\node[above] at (65/14-0.1,45/14+0.1) {$E_{2}^5$};
	\draw [thick](6,0) node[below]{$\frac45$}--(10,4); %gap condition
	\end{tikzpicture}\caption{The thick line segment represents the optimal range of $(\frac1p,\frac1q)$ for which the Carleman inequality \eqref{carl} holds when $d=5$ and $k=2$.}\label{fig1}
\end{figure}
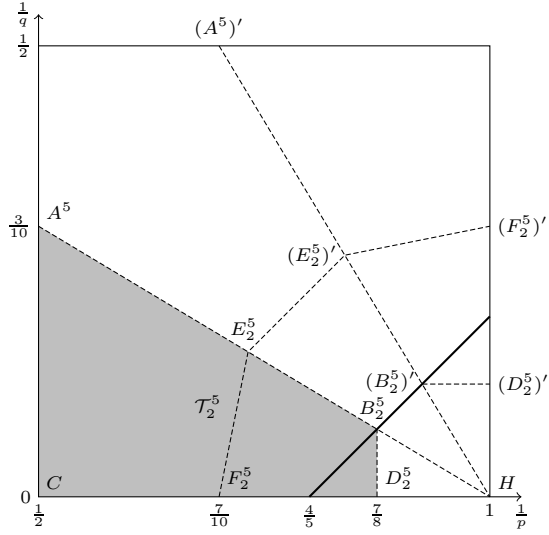
%%%%%%%%%
We first prove the sufficiency part. By duality the estimate \eqref{e:local_est} is equivalent to $\|\overline{m_L}(D)\|_{q'\to p'}\lesssim 1$. By Lemma \ref{l:adj_same} this is equivalent to $\|m_L(D)\|_{q'\to p'}\lesssim 1$. Thus, if we prove \eqref{e:local_est} for $(\frac1p,\frac1q)\in [E_{k}^d, F_{k}^d)$, then the same estimate follows for $(\frac1p,\frac1q)\in [E_{k}^d, F_{k}^d, (E_{k}^d)', (F_{k}^d)'] \setminus\{ F_k^d, (F_k^d)'\}$ by duality and interpolation. Moreover, since $m_L$ is supported in the ball of radius $2$ centered at the origin, it is easy to see $\|m_L(D)\|_{1\to \infty}\lesssim 1$. Indeed, for a cutoff function $\phi\in C^\infty_0(\R^d)$ such that $\phi=1$ on $\supp m_L$, we have $m_L(D) =(\phi^2 m_L)(D).$ So, from the Young inequality and \eqref{e:local_est} with $(\frac1p,\frac1q)\in [E^d_k,F^d_k)$ it follows that
\[ \|m_L(D)f\|_{L^\infty} \le \| \phi^\vee\|_{L^{q'}}\| (\phi m_L)(D)f\|_q\lesssim \|\phi(D)f\|_{L^p} \le \|\phi^\vee\|_{L^p}\|f\|_L^1\lesssim\|f\|_{L^1}. \]
 Again, interpolation gives \eqref{e:local_est} for all $1<p,q<\infty$ satisfying \eqref{e:penta}. Hence,  it is enough to prove \eqref{e:local_est} for $(\frac1p,\frac1q)\in [E_{k}^d, F_{k}^d)$.

%%%%%%%%%%
\begin{figure}
\captionsetup{type=figure,font=footnotesize}
\centering
\begin{tikzpicture} [scale=0.6]\scriptsize
	\path [fill=lightgray] (0,0)--(5,0)--(5,25/7)--(0,50/7)--(0,0);
	\draw (2, 3) node{$\mathcal T_{2}^7$};
	\draw [<->] (0,10.7)node[left]{$\frac1q$}--(0,0) node[below]{$\frac12$} node[left]{$0$} node[above right]{$C$}--(10.7,0) node[below]{$\frac1p$};
	\draw (0,10) node[left]{$\frac12$} --(10,10)--(10,0) node[below]{$1$} node[above right]{$H$};
	\draw [dash pattern={on 2pt off 1pt}] (0,50/7)node[above right]{$A^7$} node[left]{$\frac5{14}$}--(10,0); %the adjoint restriction line
	\draw [dash pattern={on 2pt off 1pt}] (10-50/7,10)node[above]{$(A^7)'$} --(10,0); %the restriction line
	\draw [dash pattern={on 2pt off 1pt}] (5,0)node[below]{$\frac34$} node[above right]{$D_{2}^7$}--(5,25/7); %the line BD
	\draw [dash pattern={on 2pt off 1pt}](10-25/7,10-5)--(10,10-5)node[right]{$(D_{2}^7)'$}; %the line B'D'
	\draw [dash pattern={on 2pt off 1pt}](20/7,0) node[below]{$\frac9{14}$} node[above right]{$F_{2}^7$}--(95/27,125/27)--(10-125/27,10-95/27) --(10,10-20/7) node[right]{$(F_{2}^7)'$}; %the line FEE'F'
	\draw [dash pattern={on 2pt off 1pt}] (10/7,0) node[below]{$\frac47$}--(10,10-10/7); %gap condition
	\draw [thick] (65/21,5/3)--(10-5/3, 10-65/21); 
	\node[above right] at (10-125/27-0.2,10-95/27) {$(E_{2}^7)'$};
	\node[left] at (10-25/7-0.05,10-5) {{$(B_{2}^{7})'$}};
	\node[above] at (10-5/3-0.25, 10-65/21) {$(G_{2}^7)'$};
	\node[above] at (5,25/7+0.1) {$B_{2}^7$};
	\node[above] at (95/27-0.1,125/27+0.05) {$E_{2}^7$};
	\node[left] at (65/21-0.1,5/3) {$G_{2}^7$};
\end{tikzpicture}\caption{The thick line segment represents the optimal range of $(\frac1p,\frac1q)$ for which the Carleman inequality \eqref{carl} holds when $d=7$ and $k=2$.}\label{fig2}
\end{figure}
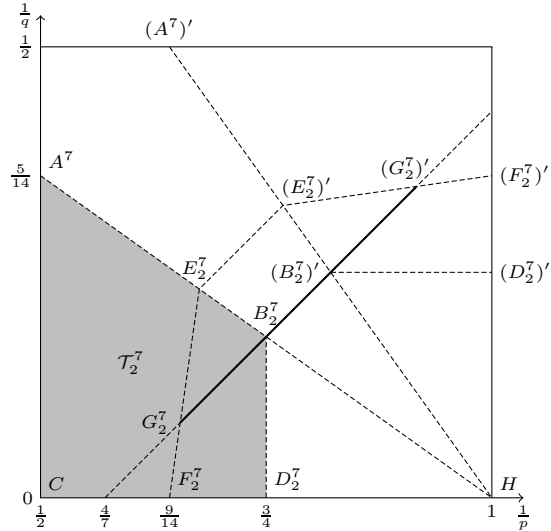
%%%%%%%%%

Let $\beta\in C_0^\infty([-4, -1/4]\cup [1/4, 4])$ be such that $\beta \psi=\psi$. Since $q\ge2$, using \eqref{e:mtau}, by the Littlewood--Paley  inequality and the Minkowski inequality we have
\begin{align*}
\|m_L(D)f\|_{L^q(\R^d)} &\sim \bigg\| \bigg( \sum_{\ep\in\mathbb D} \Big|\beta \Big(\frac{D_d}\ep\Big) m_\ep (D)f \Big|^2\bigg)^\frac12 \bigg\|_{L^q(\R^d)} \\
& \le \bigg( \sum_{\ep\in \mathbb D} \Big\| \beta \Big(\frac{D_d}\ep\Big) m_\ep (D)f  \Big\|_{L^q(\R^d)}^2 \bigg)^\frac12 .
\end{align*}
Combined with Corollary \ref{p:suff}, this gives  
\[	\|m_L(D)f\|_{L^q(\R^d)} \lesssim  \bigg( \sum_{\ep\in \mathbb D} \ep^{2(\frac dp -\frac1q -\frac{d-2+2k}2)} \Big\| \beta \Big(\frac{D_d}\ep\Big) f  \Big\|_{L^p(\R^d)}^2 \bigg)^\frac12	\]
if $(\frac1p,\frac1q)\in \mathcal T_{k}^d$. We notice that $[E_{k}^d, F_{k}^d)\subset \mathcal T_{k}^d$ and $\frac dp -\frac1q -\frac{d-2+2k}2=0$ if $(\frac1p, \frac1q)\in [E_{k}^d, F_{k}^d)$. Since $p\le 2$ the Minkowski inequality and the Littlewood--Paley inequality give
\[	\|m_L(D)f\|_{L^q(\R^d)} \lesssim \bigg\| \bigg( \sum_{\ep\in\mathbb D} \Big|\beta \Big(\frac{D_d}\ep\Big) f \Big|^2\bigg)^\frac12 \bigg\|_{L^p(\R^d)} \sim \|f\|_{L^p(\R^d)}	\]
whenever $(\frac1p, \frac1q)\in [E_{k}^d, F_{k}^d)$.

Now, we prove the necessity part. We need only show that the inequality \eqref{e:local_est} implies the first and the second inequalities in \eqref{e:penta} since the other in \eqref{e:penta} follows from the second via duality.  From the assumption \eqref{e:local_est} it follows $\|m_\ep(D)\|_{p\to q} \lesssim 1 $ uniformly in $\ep\in \mathbb D$. Hence, by Propositions \ref{p:example1} and \ref{p:example2} we have 
\[	\ep^{\frac{d+2}2(\frac1p-\frac1q)-k} \lesssim 1 , \ \  \ep^{\frac{d}p-\frac1q -\frac{d-2+2k}2} \lesssim 1	\]
for $\ep \ll 1$. Thus, considering the limiting case $\ep\to 0$ gives the first and the second inequalities in \eqref{e:penta}.
\end{proof}

\section{Unique continuation (proof of Theorem \ref{t:wucp})} \label{s:wucp}
For the differential inequality $|\Delta u|\le |Vu|$, the argument deducing the UCP from the Carleman inequality \eqref{e:carl-1} is well-known (\cite{KRS87}). One can prove Theorem \ref{t:wucp} combining Theorem \ref{t:carl} and an argument in \cite[Corollary 3.2]{KRS87}. We need only replace the Kelvin transform by an analogous point inversion transform preserving the polyharmonicity.

\begin{lemma}\label{l:kelvin}
For $0<s<\frac d2$  and $u$ supported away from the origin in $\R^d$, let  
\begin{equation}\label{e:kel}
 \mathcal T_{s}u (x) = |x|^{-d+2s} u\circ \Phi (x), \ \ x\ne 0, 
\end{equation}
where $\Phi(x) =|x|^{-2}x$. Then, we have 
\begin{equation}\label{eq_kelvin} 
(-\Delta)^s \mathcal T_{s} u(x)= |x|^{-d-2s} \big( (-\Delta)^s u\big) \circ \Phi (x),  \ \ x\ne 0.
\end{equation}
\end{lemma}
Lemma \ref{l:kelvin} is already known. For example, see \cite[Proposition 2]{ADS} and \cite[Lemma 3]{DG}. Nevertheless, we provide a short proof different from those in \cite{ADS} and \cite{DG} for the sake of completeness. 
\begin{proof}[Proof of Lemma \ref{l:kelvin}]
If we set $f=(-\Delta)^s u$, then \eqref{eq_kelvin} is equivalent to 
\begin{equation} \label{e:kelvin}
|\cdot|^{-d+2s} \big((-\Delta)^{-s}f\big)\circ \Phi = (-\Delta)^{-s} \big( |\cdot|^{-d-2s} f\circ\Phi \big).
\end{equation}
Since 
\[	(-\Delta)^{-s}f(x) = c_{d,s} \int_{\R^d} \frac{f(y)}{|x-y|^{d-2s}} dy\]
and $|\det J\Phi(y)|=|y|^{-2d}$, 
\begin{align*}
|x|^{-d+2s} \big((-\Delta)^{-s}f\big)\circ \Phi (x) 
&= c_{d,s} |x|^{-d+2s} \int_{\R^d} \frac{f(y)}{|\frac x{|x|^2} -y|^{d-2s}} dy \\
&= c_{d,s} |x|^{-d+2s} \int_{\R^d} \frac{f\circ\Phi(y)|y|^{-2d}}{|\frac x{|x|^2} - \frac{y}{|y|^2}|^{d-2s}} dy.
\end{align*}
Since $|\frac x{|x|^2} - \frac{y}{|y|^2}|^2=|x|^{-2}|y|^{-2} |x-y|^2$, this is equal to 
\[	c_{d,s} \int_{\R^d} \frac{f\circ\Phi(y) |y|^{-d-2s}}{|x-y|^{d-2s}} dy = (-\Delta)^{-s} \big( |\cdot|^{-d-2s} f\circ\Phi \big)(x),	\]
and we obtain   \eqref{e:kelvin}.
\end{proof}

%%%%%%%%%%
\begin{figure}
\captionsetup{type=figure,font=footnotesize}
\centering
\begin{tikzpicture} [scale=0.6]\scriptsize
	\path [fill=lightgray] ({20*cos(60)},20)--({20*cos(120)} ,20)--({20*cos(120)} ,791/40)--({20*cos(60)},791/40);
	\draw ({20*cos(60)},{20*sin(60)}) arc(60:120:20);
	\draw (0,20) circle [radius=6];
	\draw (0,20) circle [radius=3];
	\draw ({20*cos(60)},20)--({20*cos(120)} ,20);
	\draw ({20*cos(60)},791/40)--({20*cos(120)} ,791/40);
	\draw ({20*cos(60)}, 1591/80) node[right]{$S_\rho$};
	\draw ({20*cos(60)},{20*sin(60)}) node[below]{$S^{d-1}$};
	\draw [->] [dash pattern={on 2pt off 1pt}] (0,20)node[above]{$e_1$} --({6*cos(135)},{20+ 6*sin(135)}); 	\draw ({18/5 *cos(135)},{20+ 18/5*sin(135)}) node[above]{$\delta_1$};
	\draw [->] [dash pattern={on 2pt off 1pt}] (0,20) --({3*cos(30)},{20+ 3*sin(30)}); 	\draw ({9/5*cos(30)},{20+ 9/5*sin(30)}) node[above]{$\frac{\delta_1}2$};
	\node at (0,20) [circle,fill,inner sep=0.7pt]{};
\end{tikzpicture}\caption{Continuation of $u=0$ from $A_+(\delta_1)$ to $S_\rho\cap B(e_1,\frac{\delta_1}2)$ in the proof of Lemma \ref{lem_ucp}.}\label{fig_ucp}
\end{figure}
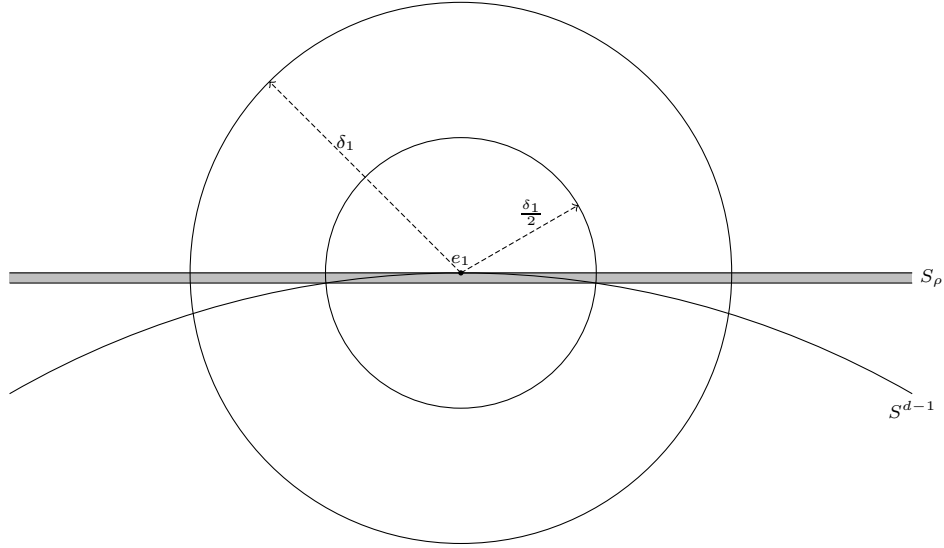
%%%%%%%%%%%

Using \eqref{carl}, Lemma \ref{l:kelvin}, and the argument in \cite[p. 345]{KRS87}, we obtain the following.
\begin{lemma}\label{lem_ucp}
Let $d\ge 3$ and $\Omega$ be a connected open set in $\R^d$ containing $S^{d-1}$. Let $k$, $V$, and $\mathbf X$ be given as in Theorem \ref{t:wucp}. Suppose that  $u\in \mathbf X$ satisfies \eqref{e:diff_ineq}. If $u$ vanishes in the annulus  $A_+(\delta_1):=\{x\in \R^d \colon 1<|x|<1+\delta_1\}$ for some $\delta_1>0$, then $u$ also vanishes in $A_-(\delta_2):=\{ x\in \R^d \colon 1-\delta_2<|x|<1\}$ for some $0<\delta_2<1$, and vice versa.
\end{lemma}
\begin{proof}
First, we assume $u=0$ in $A_+(\delta_1)$ and show that $u$ vanishes in $A_-(\delta_2)$ for some $0<\delta_2<1$. By rotational symmetry, it is enough to prove that $u=0$ in a neighborhood of $e_1:=(1,0,\cdots,0)\in \R^d$.

Let $p>1$ if $k\ge\frac{d-2}2$ and  $p= \frac{2(d-1)}{d+2k}$ if $k<\frac{d-2}2$, and let $q\in (1,\infty)$ be such that \eqref{e:cond} holds. We denote by $B(x,r)$ the open ball in $\R^d$ of radius $r$ centered at $x$. We choose a cutoff function $\phi \in C^\infty_0(B(0,\delta_1))$ satisfying $\phi=1$ in $B(0,\frac{\delta_1}2),$ and define $\tilde u(x)=u(x)\phi(x-e_1)$. Since $V\in L^{d/2k}_{\mathrm{loc}}(\Omega)$ and $B(0,1)$ is convex, there is a small $\rho>0$ such that 
\begin{equation}\label{eq_V}
 C\|V\|_{L^{\frac{d}{2k}} ( S_\rho \cap B(e_1, \frac{\delta_1}2))} \le \tfrac12 \quad \text{and}\quad  B(0,1) \cap S_\rho \subset B(e_1, \tfrac{\delta_1}2),
 \end{equation}
where $S_\rho:=\{x\in\R^d\colon 1-\rho< x_1\le 1\}$ and $C$ is the constant in the Carleman inequality \eqref{carl}. See Figure \ref{fig_ucp}.

Obviously, 
\[	\supp \tilde u \subset B(e_1,\delta_1)\setminus A_+(\delta_1)\subset \big[ S_\rho \cap \big(B(e_1, \delta_1)\setminus A_+(\delta_1)\big) \big] \cup \{x_1\le 1-\rho\}.	\] 
So, by the inclusion in \eqref{eq_V} we have 
\[	S_\rho \cap \big(B(e_1, \delta_1)\setminus A_+(\delta_1)\big)\subset S_\rho \cap B(0,1) \subset S_\rho \cap B(e_1, \tfrac{\delta_1}2).\] 
Therefore, the Carleman inequality \eqref{carl} gives 
\begin{equation}\label{e:car1}
\|e^{\lambda x_1} \tilde u\|_{L^q(S_\rho)} 
\le C \|e^{\lambda x_1} \Delta^k \tilde u\|_{L^p(S_\rho \cap B(e_1,\frac{\delta_1}2))} + C \|e^{\lambda x_1} \Delta^k \tilde u\|_{L^p(\{x_1\le 1-\rho\})}
\end{equation}
for all $\lambda>0$. Since $\tilde u =u$ in $B(e_1, \frac{\delta_1}2)$, by \eqref{e:diff_ineq}, H\"older's inequality, and \eqref{eq_V} 
\begin{equation}\label{e:car2}
\begin{aligned}
C \|e^{\lambda x_1} \Delta^k \tilde u\|_{L^p(S_\rho \cap B(e_1,\frac{\delta_1}2))}  
&\le C\|e^{\lambda x_1} V u\|_{L^p(S_\rho \cap B(e_1,\frac{\delta_1}2) )} \\
&\le C\|V\|_{L^{\frac{d}{2k}} ( S_\rho \cap B(e_1, \frac{\delta_1}2))} \|e^{\lambda x_1} u\|_{L^q(S_\rho \cap B(e_1,\frac{\delta_1}2))} \\
&\le \tfrac12 \|e^{\lambda x_1} \tilde u\|_{L^q(S_\rho)} .
\end{aligned}
\end{equation}
Since $\tilde u \in W^{2k,p}(\R^d)$, it is clear that $\|e^{\lambda x_1} \tilde u\|_{L^q(S_\rho)}<\infty$ by the Sobolev embedding. Therefore, combining the estimates \eqref{e:car1} and \eqref{e:car2}, we have
\[	\|e^{\lambda x_1} \tilde u\|_{L^q(S_\rho)} \le 2C\|e^{\lambda x_1} \Delta^k \tilde u\|_{L^p(\{x_1< 1-\rho\})} \le 2Ce^{\lambda (1-\rho)} \|\Delta^k \tilde u\|_{L^p(\{x_1\le1-\rho \})},	\]
equivalently,
\[	\|e^{\lambda (x_1-1+ \rho )} u\|_{L^q(S_\rho\cap B(e_1,\frac{\delta_1}2))} \le 2C \|\Delta^k \tilde u\|_{L^p(\{x_1\le1-\rho\})}	\]
uniformly in $\lambda$. The right hand side is clearly finite, so the inequality leads to a contradiction as $\lambda\to \infty$ unless $u$ were identically zero in $S_\rho$.
 
We next prove the converse  assertion, that is, we show $u=0$ in a neighborhood of $e_1$ assuming $u=0$ in $A_-(\delta_2)$. 

Making use of the Kelvin transform \eqref{e:kel} we notice that $\CT_ku =0$ in $A_+(\delta_1)$ with $\delta_1=\frac{\delta_2}{1-\delta_2}$. If we set  $V^*(x)=|x|^{-4k}V\circ \Phi(x)$, then for every $K\subset\subset \Phi(\Omega)$
\[	\|V^\ast\|_{L^\frac{d}{2k}(K)} = \|V\|_{L^\frac{d}{2k}(\Phi^{-1}(K))} <\infty,	\]
so $V^*\in L^{d/2k}_{\mathrm{loc}}(\Omega)$. It follows from \eqref{eq_kelvin} and \eqref{e:diff_ineq} that 
 \[ |\Delta^k \CT_ku (x)|=|x|^{-d-2k}| (\Delta^k u)\circ \Phi(x)| \le |V^*(x)\CT_k u(x)|.\]
Now, we may use the preceding argument proving continuation of $u=0$ from the outside of $S^{d-1}$ to the inside. Thus we see $\CT_k u$ vanishes near $e_1$. Consequently, $u$ also vanishes in a neighborhood of $e_1.$
\end{proof}

By the argument in \cite[p. 344]{KRS87} Theorem \ref{t:wucp} is rather a straightforward consequence of the above lemma. 
\begin{proof}[Proof of Theorem \ref{t:wucp}] 
Let $u\in\mathbf{X}$ vanish in a nonempty open subset of $\Omega$, and we denote by $\Omega_0$ the maximal open set in which $u=0$. We aim to prove $\Omega_0=\Omega$. 

Suppose to the contrary that $\Omega_0\neq\Omega$. Then there exists an $x_0\in \Omega\cap \partial \Omega_0$ and an open ball $B\subset \Omega_0$ such that $\partial B\cap \partial \Omega_0 =\{x_0\}$. By translation and scaling, we may assume $B=B(0,1)$. Since $u=0$ in $B$, Lemma \ref{lem_ucp} shows that $u=0$ in a neighborhood of $x_0$. This contradicts to the maximality of $\Omega_0$. Therefore, $\Omega_0=\Omega$.
\end{proof}

\section{Proof of Proposition \ref{p:example2}} \label{s:pf_lower}
In this section, we prove \eqref{eq_goal} by constructing examples. In the first place, it would be convenient to define notations for quantities relevant to the example and obtain estimates for those quantities.  

For $\frac12 \le \tau \le 2$, $y\in \R^{d-1}$, $0<\ep \le 1$, and a nonnegative $\varphi\in C_0^\infty(\R)$ supported near $1$,  let us define 
\begin{equation}\label{eq_Ij}
\begin{aligned}
\CI_1(\tau,y;\ep) &:= \int_{\R} \frac{2\ep\tau}{(\rho^2 -1 +\ep^2\tau^2)^2 +4\ep^2\tau^2}  \varphi(\rho) \cos((\rho-1)|y|)\,d\rho,\\
\CI_2(\tau,y;\ep) &:= \int_{\R} \frac{\rho^2 -1 +\ep^2\tau^2}{(\rho^2 -1 +\ep^2\tau^2)^2 +4\ep^2\tau^2}  \varphi(\rho) \cos((\rho-1)|y|)\,d\rho,\\
\CI_3(\tau,y;\ep) &:= \int_{\R} \frac{2\ep\tau}{(\rho^2 -1 +\ep^2\tau^2)^2 +4\ep^2\tau^2}  \varphi(\rho) \sin((\rho-1)|y|)\,d\rho,\\
\CI_4(\tau,y;\ep) &:= \int_{\R} \frac{\rho^2 -1 +\ep^2\tau^2}{(\rho^2 -1 +\ep^2\tau^2)^2 +4\ep^2\tau^2}  \varphi(\rho) \sin((\rho-1)|y|)\,d\rho.
\end{aligned}
\end{equation}
We also set
\begin{align*}
\wt\CI_1(\tau;\ep) &:= \int_{\R} \frac{2\ep\tau}{(\rho^2 -1 +\ep^2\tau^2)^2 +4\ep^2\tau^2}  \varphi(\rho) \,d\rho,\\
\wt\CI_2(\tau;\ep) &:= \int_{\R} \frac{\rho^2 -1 +\ep^2\tau^2}{(\rho^2 -1 +\ep^2\tau^2)^2 +4\ep^2\tau^2}  \varphi(\rho) \,d\rho.
\end{align*}
By a simple calculation we get the following.
\begin{lemma}\label{lem_Ij} 
 For a fixed $\delta_\circ \in (0,\frac14)$ let $\varphi\in C^\infty_0([1-2\delta_\circ, 1+2\delta_\circ])$. Then there is a constant $C$ independent of $0<\ep\ll 1$ such that
\[	 |\wt \CI_1(\tau;\ep)|  \le C \ \ \text{and} \ \  |\wt \CI_2(\tau;\ep)| \le C \log \ep^{-1}.	\]
Consequently, 
\begin{align*}
|\CI_1(\tau, y;\ep)| , \, |\CI_3(\tau, y;\ep)| &\le C, \\
|\CI_2(\tau,y;\ep)|, \, |\CI_4(\tau,y;\ep)|  &\le C \log\ep^{-1}.
\end{align*}
\end{lemma}
\begin{proof}
Since the function $\rho\mapsto \rho^2 -1+\ep^2\tau^2$ is invertible on $[1-2\delta_\circ, 1+2\delta_\circ]$, by changing variables, it is easy to see 
\[	|\wt\CI_1(\tau;\ep)|
	\le \|\varphi\|_{L^\infty} \int_{-6}^6 \frac{2\ep\tau}{t^2 +4\ep^2\tau^2} dt 
	\lesssim  \int_0^\infty \frac1{t^2+1} dt
	\lesssim 1.	\]
Similarly, we get
\begin{align*}
|\wt\CI_2(\tau;\ep)|
\le \|\varphi\|_{L^\infty} \int_{-6}^{6}  \frac{|t|}{t^2 +4\ep^2\tau^2} dt 
\lesssim \int_{0}^{1/\ep}  \frac{t}{t^2 +1} dt
\lesssim  \log\ep^{-1}.
\end{align*}
The estimates for $\CI_j(\tau, y;\ep)$ follow in a similar manner. 
\end{proof}

Fortunately, under additional assumptions on $\varphi$ we can improve the estimates for $\CI_2$ and $\CI_4$.  We also obtain a proper lower bound for $\CI_1$.

\begin{lemma} \label{lem_add}
Let $\delta_\circ$, $\ep$, and $\varphi$ be given as in Lemma \ref{lem_Ij}, and suppose further that $0\le \varphi(\rho)\le1$ and $\varphi(1+\rho)=\varphi(1-\rho)$ for $\rho \in\R$, and   
\begin{equation}\label{e:varphi1} 
\varphi(\rho) =1\ \ \text{if} \ \ \rho\in [1-\delc,1+\delc].
\end{equation}
Then, we have the uniform estimates 
\begin{align}
\label{e:I2upper}
|\CI_2(\tau,y;\ep)| &\lesssim 1 \ \ \text{for all} \ \ \tau\in [1/2,2] \ \ \text{and} \ \ y\in \R^{d-1}, \\
\label{e:I4upper}
|\CI_4(\tau,y;\ep)| &\lesssim 1 \ \ \text{for all} \ \ \tau\in [1/2,2] \ \ \text{and} \ \ |y|\sim \ep^{-1}.
\end{align}
Furthermore, 
\begin{equation}\label{lower_I1}
 \CI_1(\tau,y;\ep)\gtrsim 1 \ \ \text{for all} \ \ \tau\in [1/2,2] \ \ \text{and} \ \ |y|\sim \ep^{-1}.
\end{equation}
\end{lemma}
\begin{proof}
In order to use the symmetry of the function $\varphi$ we approximate $\rho^2 -1 +\ep^2\tau^2$ by $2(\rho-1)$ and write
\begin{align*}
&\frac{\rho^2 -1 +\ep^2\tau^2}{(\rho^2 -1 +\ep^2\tau^2)^2 +4\ep^2\tau^2}   \\
&=\frac{\rho-1}{2(\rho-1)^2 +2\ep^2\tau^2} +\Big( \frac{\rho^2 -1 +\ep^2\tau^2}{(\rho^2 -1 +\ep^2\tau^2)^2 +4\ep^2\tau^2} - \frac{\rho-1}{2(\rho-1)^2 + 2\ep^2\tau^2}\Big)\\
&=\frac{\rho-1}{2(\rho-1)^2 +2\ep^2\tau^2} - \frac{ (\rho-1)^4 (\rho+1) + 2\ep^2\tau^2 (\rho-1)^3 +\ep^4\tau^4 (\rho-3)}{2\big((\rho-1)^2 +\ep^2\tau^2\big) \big((\rho^2 -1 +\ep^2\tau^2)^2 +4\ep^2\tau^2 \big)}\\
&=: I_1(\rho,\tau;\ep) - I_2(\rho,\tau;\ep).
\end{align*}
Since the function $\rho\mapsto \rho\varphi(1+\rho)$ is  odd, we see that
\[	\int I_1(\rho,\tau;\ep) \varphi(\rho)\cos((\rho-1)|y|)\,d\rho 
= \int_{-2\delc}^{2\delc} \frac{\rho\varphi(1+\rho)}{2(\rho^2+\ep^2\tau^2)} \cos(\rho|y|) d\rho =0,	\]
and 
\[	\CI_2(\tau,y;\ep) = -\int I_2(\rho,\tau;\ep) \varphi(\rho) \cos((\rho-1)|y|)\,d\rho.	\] 
To estimate this we separately consider the three terms in the numerator in $I_2$. 

First, by translation $\rho\to\rho+1$
\begin{align*}
&\Big|\int \frac{ (\rho-1)^4 (\rho+1) \varphi(\rho) \cos((\rho-1)|y|) }{2\big((\rho-1)^2 +\ep^2\tau^2\big) \big((\rho^2 -1 +\ep^2\tau^2)^2 +4\ep^2\tau^2\big)}  \,d\rho \Big|\\
&\lesssim  \int_{-2\delc}^{2\delc} \frac{ \rho^2 (\rho+2) }{(\rho(\rho+2) +\ep^2\tau^2)^2 +4\ep^2\tau^2}\,d\rho .
\end{align*}
If $|\rho|\ge \ep^2\tau^2$ then $(\rho(\rho+2) +\ep^2\tau^2)^2=\rho^2((\rho+2) +\frac{\ep^2\tau^2}\rho)^2 \ge\rho^2(\rho+1 )^2$, hence 
\begin{align*}
& \int_{-2\delc}^{2\delc} \frac{ \rho^2 (\rho+2) }{(\rho(\rho+2) +\ep^2\tau^2)^2 +4\ep^2\tau^2}\,d\rho \\
& \le \int_{\ep^2\tau^2\le |\rho|\le 2\delc} \frac{ \rho+2 }{(\rho+1)^2}\,d\rho + \int_{|\rho|\le \ep^2\tau^2} \frac{ \rho^2 (\rho+2) }{4\ep^2\tau^2}\,d\rho \lesssim 1.
\end{align*}
Secondly, similar computations show that  
\begin{align*}
&\Big|\int \frac{ \ep^2\tau^2 (\rho-1)^3 \varphi(\rho) \cos((\rho-1)|y|) }{\big((\rho-1)^2 +\ep^2\tau^2\big) \big((\rho^2 -1 +\ep^2\tau^2)^2 +4\ep^2\tau^2\big)}  \,d\rho \Big| \\
&\lesssim \ep^2 \int_{-2\delc}^{2\delc} \frac{ |\rho|}{ (\rho(\rho+2) +\ep^2\tau^2)^2 +4\ep^2\tau^2 } \,d\rho\\
&\lesssim \ep^2  \Big( \int_{\ep^2\tau^2 \le |\rho|\le 2\delc} \frac{|\rho|}{\rho^2 (\rho+1)^2} d\rho + \int_{|\rho|\le \ep^2\tau^2} \frac{|\rho|}{4\ep^2\tau^2} d\rho \Big) \\
&\lesssim \ep^2 ( \log \ep^{-2} + \ep^2 ) \lesssim 1.
\end{align*}
Finally, we see that
\begin{align*}
&\Big|\int \frac{ \ep^4\tau^4 (\rho-3) \varphi(\rho) \cos((\rho-1)|y|)}{2\big( (\rho-1)^2 +\ep^2\tau^2 \big) \big((\rho^2 -1 +\ep^2\tau^2)^2 +4\ep^2\tau^2 \big)}  \,d\rho \Big| \\
& \lesssim \ep^4 \int_{-2\delc}^{2\delc} \frac{2-\rho}{(\rho^2 +\ep^2\tau^2) \big((\rho(\rho+2) +\ep^2\tau^2)^2 +4\ep^2\tau^2\big)} \,d\rho  \lesssim 1.
\end{align*}
Combining these estimates, we obtain the uniform estimate \eqref{e:I2upper}.

Let us prove the estimate \eqref{e:I4upper}. As before we can write
\[	\CI_4(\tau,y;\ep) = \int \big(I_1(\rho,\tau;\ep)-I_2(\rho,\tau;\ep)\big) \varphi(\rho) \sin((\rho-1)|y|)\,d\rho.	\] 
Using the argument identical to one in the above we have 
\[ \Big| \int I_2(\rho,\tau;\ep) \varphi(\rho) \sin((\rho-1)|y|) d\rho \Big| \lesssim 1\]
uniformly in $\tau\in [1/2,2]$ and $y\in \R^d$.  Thus it remains to prove
\begin{equation}\label{I4}
\Big| \int I_1 (\rho,\tau;\ep) \varphi(\rho) \sin ((\rho-1)|y|) \, d\rho \Big| \lesssim 1
\end{equation}
uniformly in $\tau\in [1/2,2]$ and $|y|\sim \ep^{-1}$. By the assumption on $\varphi$ the integral in \eqref{I4} is equal to
\[
\int_{-2\delc}^{2\delc} \frac{\rho \varphi(\rho+1) \sin(\rho|y|)}{2(\rho^2 +\ep^2\tau^2)}  \,d\rho
= \int_{0}^{2\delc} \frac{\rho \varphi(\rho+1) \sin(\rho|y|)}{\rho^2 +\ep^2\tau^2} \,d\rho.
\]
We  break this into the following two terms:
\begin{align*}
\CI_4^1(y)&:=\int_{0}^{2\delc} \frac{\sin(\rho|y|) }{\rho}  \,d\rho, \\
\CI_4^2 (\tau,y;\ep) &:=\int_{0}^{2\delc} \Big(\frac{\rho\varphi(\rho+1)}{\rho^2 +\ep^2\tau^2}  -\frac 1\rho\Big)\sin(\rho|y|) \,d\rho. 
\end{align*}
Since $|\int_0^u \frac{\sin t}t dt|\le 4$ for $u>0$ we see  $|\CI_4^1 (y) |\le 4$ for any $y.$  From \eqref{e:varphi1} it follows that
\begin{equation}\label{I42}
 \CI_4^2 (\tau,y;\ep) = -\int_{0}^{\delc} \frac{\ep^2\tau^2 \sin(\rho|y|)}{\rho(\rho^2 +\ep^2\tau^2)} \,d\rho +\int_{\delc}^{2\delc} \Big(\frac{\rho \varphi(\rho+1)}{\rho^2 +\ep^2\tau^2}  -\frac 1\rho\Big)\sin(\rho|y|) \,d\rho.
 \end{equation}
The (absolute value of the) first integral in \eqref{I42} is dominated by
\begin{align*}
\ep^2\tau^2|y|^2 \Big| \int_0^{\delc |y|} \frac {\sin \rho }{\rho(\rho^2 +\ep^2\tau^2 |y|^2)} d\rho \Big| 
\le \ep\tau |y| \int_0^{\frac{\delc}{\ep\tau}} \frac 1{1+\rho^2} d\rho,
\end{align*}
which is uniformly bounded provided that $|y|\sim \ep^{-1}$. On the other hand, the second integral in \eqref{I42} is estimated by $\int_{\delc}^{2\delc} \frac{d\rho}{\rho} \lesssim 1.$ Thus we obtain \eqref{I4}, which yields the uniform bound \eqref{e:I4upper}.

Finally, we prove \eqref{lower_I1}. Let us set 
\[	b(u)= \int_{|s|\le u} \frac{1}{s^2+1}ds, \  \  0\le u\le \infty. 	\]
Clearly, this is a continuous, monotonically increasing, and bounded function on the interval $[0,\infty]$. Let us fix a large $\lambda>0$ such that 
\begin{equation}\label{e:lamb}
b(\lambda/4) \ge 2^4 \big( b(\infty)-b(\lambda/4) \big),
\end{equation}
and then let us take a small number $\mu>0$  such that $\lambda\mu\le 2^{-7}$ and define
\[	\mathbb A=\Big \{y\in \R^{d-1}\colon \frac\mu{4\ep} \le |y| \le \frac\mu{2\ep} \Big\}.	\]
We now break the integral $\mathcal I_1(\tau,y;\ep)$ into two parts as 
\begin{align*}
\mathcal I_1(\tau,y;\ep)
	&= \int_{|\rho^2-1+\ep^2\tau^2|\le\lambda\ep} \frac{2\ep \tau \varphi(\rho) \cos ((\rho-1)|y|)}{(\rho^2-1+\ep^2\tau^2)^2+(2\ep\tau)^2} d\rho \\
	&\quad + \int_{|\rho^2-1+\ep^2\tau^2|\ge\lambda\ep} \frac{2\ep \tau \varphi(\rho) \cos ((\rho-1)|y|)}{(\rho^2-1+\ep^2\tau^2)^2+(2\ep\tau)^2} d\rho  \\
	&=: \mathcal I_{1}^1 (\tau,y;\ep) +\mathcal I_{1}^2 (\tau,y;\ep).
\end{align*}
Let us first estimate a lower bound for $\mathcal I_{1}^1(\tau,y;\ep)$.  Since $\delta_\circ$ and $\ep$ are small, if $|\rho^2-1+\ep^2\tau^2| \le \lambda \ep $, then it is easy to see that\footnote{In fact, it is vacuously true that $\ep\le \frac\lambda 4$ since we are assuming $\ep\ll1$ and $\lambda$ is a large number. It follows, since $\rho\in \supp\varphi \subset [1-2\delc, 1+2\delc]$ and $\delc<\frac14$, that $|\rho-1|\le \frac{\ep(\lambda+4\ep)}{\rho+1}\le \frac{2\ep\lambda}{2(1-\delc)} \le \frac43 \ep\lambda.$}  $|\rho-1|\le 2\lambda\ep$.  Hence whenever $y\in \mathbb A$  
\[	|(\rho-1)y|\le \lambda\mu \le 2^{-7}, \] 
and this yields 
\[	\cos ((\rho-1)|y| ) \ge 1- \frac{(\rho-1)^2|y|^2}2\ge 1-2^{-15}.	\]
If $\ep\le \frac{\delc}{2\lambda}$ then $|\rho-1|\le 2\lambda\ep \le \delc$, so $\varphi(\rho)=1$ by the assumption \eqref{e:varphi1}. Thus, for $\ep\le \frac{\delc}{2\lambda}$, $\tau\in[1/2,2]$, and $y\in\mathbb A$, we have
\begin{align*}
	\mathcal I_{1}^1 (\tau,y;\ep) 
	&\ge (1-2^{-15}) \int_{|\rho^2-1+\ep^2\tau^2|\le \lambda \ep} \frac{2\ep \tau }{(\rho^2-1+\ep^2\tau^2)^2+(2\ep\tau)^2} \frac{\varphi(\rho)}{2\rho} 2\rho d\rho \\
	&\ge \frac{1-2^{-15}}{2(1+\delc)} \int_{|t|\le \lambda\ep} \frac{2\ep\tau}{t^2+(2\ep\tau)^2} dt \ge \frac{1-2^{-15}}{4} \int_{|t|\le \frac{\lambda}{2\tau}} \frac{1}{t^2+1} dt \\
	&\ge 2^{-3}  b(\lambda /4 ) \ge 2 \big( b(\infty)-b(\lambda/4) \big)
\end{align*}
by the choice of $\lambda$ (see \eqref{e:lamb}). Meanwhile, since $\frac{\varphi(\rho)}{2\rho}\le \frac{\varphi(\rho)}{2(1-2\delc)} \le \varphi(\rho) \le 1$,
\begin{align*}
	|\mathcal I_{1}^2 (\tau,y;\ep) |
	&\le \int_{|\rho^2-1+\ep^2\tau^2|\ge \lambda \ep} \frac{2\ep \tau }{(\rho^2-1+\ep^2\tau^2)^2+(2\ep\tau)^2} \frac{\varphi(\rho)}{2\rho} 2\rho d\rho \\
	&\le \int_{|t|\ge \lambda \ep} \frac{2\ep \tau}{t^2+(2\ep\tau)^2} dt 
	=  \int_{|s|\ge\frac\lambda{2\tau}} \frac1{s^2+1}ds \\
	&\le b(\infty)-b(\lambda /4 ).
\end{align*}
Combining the estimates for $\CI_1^1$ and $\CI_1^2$ we have, for all $(\tau,y)\in [1/2,2]\times \mathbb A$ and $\ep$ small enough, 
\[	\mathcal I_{1}(\tau, y; \ep)\ge \,\mathcal I_1^1 (\tau, y; \ep)-|\mathcal I_{1}^2 (\tau, y; \ep)| \ge b(\infty)-b(\lambda /4 ). 	\]
Therefore, the proof of \eqref{lower_I1} is complete. 
\end{proof}

Now we prove Proposition \ref{p:example2}.
\begin{proof}[Proof of Proposition \ref{p:example2}]
It is enough to prove \eqref{eq_goal}. Let us fix a small enough $\delc>0$ and choose a nonnegative smooth function  $\phi\in C^\infty_0([1-2\delc,1+2\delc])$ such that 
\begin{itemize}
\item the function $\rho\mapsto(\rho+1)^\frac{d-2k}2\phi(\rho+1)$ is even;
\item $\sup_\rho\rho^{\frac{d-2k}2}\phi(\rho)\le 1$;
\item $\rho^{\frac{d-2k}2} \phi(\rho)=1$  on $[1-\delc, 1+\delc]$.
\end{itemize}
We define $f\in C_0^\infty(\R^d)$ by 
\[	\widehat f(\eta,\tau)=\frac{\phi(|\eta|)\phi(\tau)}{\psi_0(\ep_\circ^{-1}(1-|\eta|^2))\psi(\tau)}, \ \ (\eta, \tau)\in \R^{d-1}\times \R.	\] 
Then, it is clear that $\|f\|_{L^p(\R^d)} \lesssim 1$ for every $p$, and we have by the spherical coordinates in $\R^{d-1}$ 
\begin{align*}
\widetilde m_\ep(D)f(y,t) 
	&= \frac1{(2\pi)^d} \int e^{it\tau} \phi(\tau) \int \frac{\rho^{d-2} \phi(\rho)  \wh{d\sigma_{d-2}}(\rho y)}{(\rho^2-1+\ep^2\tau^2 + 2\ep\tau i )^k}   d\rho d\tau.
\end{align*}

Integration by parts reduces the order $k$ to $1$ in the denominator. In fact, we notice that for any integer $k\ge 2$,
\[ \frac1{(\rho^2-1+\ep^2\tau^2 + 2\ep\tau i )^k} = -\frac{1}{2\rho(k-1)} \partial_\rho \frac1{(\rho^2-1+\ep^2\tau^2 + 2\ep\tau i )^{k-1}}.\]
Thus, integrating by parts $(k-1)$-times with respect to $\rho$, we have 
\begin{equation}\label{eq_ibp}
\widetilde m_\ep(D)f(y,t) = c
\int e^{it\tau} \phi(\tau) \int \frac{\Phi(\rho,y)}{\rho^2-1+\ep^2\tau^2 + 2\ep\tau i } d\rho d\tau
\end{equation}
for some constant\footnote{Throughout the proof, we let $c$ denote some positive constant which may differ at each occurrence depending only on $k$ or $d$.} $c$,  where 
\[ \Phi(\rho,y) = T^{k-1} \big(\rho^{d-2} \phi(\rho)  \wh{d\sigma_{d-2}}(\rho y)\big) \  \ \text{with} \ \ Th:= \partial_\rho (\rho^{-1}h ).  \]
Let us recall the identity (see, for example, \cite[Appendix B]{Gra14a})
\[	\wh{d\sigma_{d-2}}(\rho y) = c |\rho y|^\frac{3-d}2 J_{\frac{d-3}2} (|\rho y|), \]
where  $J_{\nu}(r)$ denotes the Bessel function of the first kind and has the following property  (for $\RE \nu>-1/2$)
\[	\partial_r \big(r^{-\nu} J_\nu(r) \big) =-r^{-\nu} J_{\nu+1}(r), \ \  r>0.	\]
Making use of these identities and the chain rule $\partial_\rho=|y|\partial_r$, $r=\rho|y|$, we have
\begin{equation}\label{eq_Phi}
\Phi(\rho,y)=\sum_{l=0}^{k-1} |\rho y|^{\frac{3-d}2-l} J_{\frac{d-3}2 +l}(\rho|y|) \phi_l(\rho)|y|^{2 l}
\end{equation}
with $\phi_l\in C_0^\infty([1-2\delc,1+2\delc])$. In particular $\phi_{k-1}(\rho)=c\rho^{d-2}\phi(\rho)$.

Setting
\[	\CJ_l(y,t)= |y|^{\frac{3-d}2+l} \int e^{it\tau} \phi(\tau) \int \frac{ \rho^{\frac{3-d}2-l} \phi_l(\rho)}{\rho^2-1+\ep^2\tau^2 + 2\ep\tau i} \, J_{\frac{d-3}2+l} (|\rho y|) d\rho d\tau,	\]
we may rewrite \eqref{eq_ibp} as
\begin{equation}\label{eq_rewrite}
\wt m_\ep(D) f(y,t) = c \sum_{l=0}^{k-1}\CJ_l (y,t).
\end{equation}
In the rest of the proof, we aim to show that 
\begin{equation}\label{Jk}
|\CJ_{k-1}(y,t)|\gtrsim \ep^{\frac d2-k}
\end{equation}
on a subset\footnote{See \eqref{e:set_lower} below for a precise description.} of $\{(y,t)\in \R^{d-1}\times \R \colon |y|\sim \ep^{-1}, |t|\ll 1\}$, whereas for every $s>0$
\begin{equation}\label{Jl}
|\CJ_l(y,t)|\lesssim \ep^{\frac{d}2-1-l-s}, \ \ 0\le l \le k-2 .
\end{equation}
Hence, in this set, $\CJ_{k-1}$ would be the leading term and dominate the others in the summation  \eqref{eq_rewrite}.

First, let us analyze $\CJ_{k-1}$. We shall use the asymptotic (see \cite[p. 580]{Gra14a}) 
\[ J_{\frac{d-3}2+{k-1}}(\rho|y|) =\sqrt{\frac2{\pi \rho |y| }} \cos\Big(\rho|y|-\frac{\pi(d+2k-4)}4\Big) +R(\rho|y|), \]
where $|R(r)|\lesssim r^{-\frac32}$ for $r\ge1$, and the formula 
\[	\cos(u+v)=\cos u \cos v-\sin u \sin v	\]
with $u=(\rho-1)|y|$ and $v=|y|- \frac\pi 4 (d+2k-4)$. Using these we break $\CJ_{k-1}$ as 
\begin{align*}
\CJ_{k-1}(y,t) &= |y|^{\frac{2k-d+1}2} \int e^{it\tau} \phi(\tau) \int \frac{ c \rho^{\frac{d-2k+1}2} \phi (\rho)}{\rho^2-1+\ep^2\tau^2 + 2\ep\tau i} \, J_{\frac{d-3}2+l} (|\rho y|) d\rho d\tau \\
&= c\bigg( \sqrt{\frac{2}\pi}\CJ_{k-1}^1- \sqrt{\frac{2}\pi}\CJ_{k-1}^2 +\CJ_{k-1}^3 \bigg)
\end{align*}
where, setting $\alpha_{d,k}:=\frac{\pi} 4(d+2k-4)$,  we define
\begin{align*}
\CJ_{k-1}^1&:= |y|^{\frac{2k-d}2} \cos( |y|-\alpha_{d,k} )\int e^{it\tau} \phi(\tau) \int \frac{ \rho^{\frac{d-2k}2} \phi(\rho) \cos((\rho-1)|y|)}{\rho^2-1+\ep^2\tau^2 + 2\ep\tau i}\, d\rho d\tau,\\
\CJ_{k-1}^2&:= |y|^{\frac{2k-d}2} \sin ( |y|-\alpha_{d,k} )\int e^{it\tau} \phi(\tau) \int \frac{ \rho^{\frac{d-2k}2}\phi(\rho) \sin((\rho-1)|y|)}{\rho^2-1+\ep^2\tau^2+ 2\ep\tau i} \, d\rho d\tau,\\
\CJ_{k-1}^3&:= |y|^{\frac{2k+1-d}2} \int e^{it\tau} \phi(\tau) \int \frac{ \rho^{\frac{d-2k+1}2}\phi(\rho) }{\rho^2-1+\ep^2\tau^2  +2\ep\tau i} R (\rho |y|)d\rho d\tau.
\end{align*}
Regarding the lower bound for $|\CJ_{k-1}^1|$ we estimate that for $|\IM \CJ_{k-1}^1|$. Setting 
\begin{equation}\label{e:dvarphi}
	\varphi(\rho)=\rho^{\frac{d-2k}2}\phi(\rho)
\end{equation}
and  using the definitions in \eqref{eq_Ij}, we can write  
\[	\IM \CJ_{k-1}^1
	=|y|^{\frac{2k-d}2} \cos( |y|-\alpha_{d,k}) \int \phi(\tau)\big(\sin(t\tau) \CI_2(\tau,y;\ep) - \cos(t\tau) \CI_1(\tau,y; \ep) \big) d\tau .	\]
By the definition of $\phi$ the function $\varphi$ defined in \eqref{e:dvarphi} satisfies the assumption in Lemma \ref{lem_add}. By the lemma if $\ep$ is small enough, then we obtain the bounds 
\[\CI_1(\tau,y; \ep)\gtrsim 1\quad \text{and} \quad |\CI_2(\tau,y; \ep)|\lesssim 1 \]
whenever $c_1\ep^{-1} \le |y| \le c_2\ep^{-1}$ for some constants $c_1$ and $c_2$. Thus, 
if $|t|$ is sufficiently small and $c_1\ep^{-1} \le |y| \le c_2\ep^{-1}$, then we conclude that 
\begin{equation}\label{e:1}
	|\CJ_{k-1}^1(y,t)| \ge |\IM \CJ_{k-1}^1 (y,t)| \gtrsim |y|^{\frac{2k-d}2} \cos ( |y|- \alpha_{d,k} ) .	
\end{equation}
On the other hand, using the definitions \eqref{eq_Ij} and \eqref{e:dvarphi},
\[ \CJ_{k-1}^2= |y|^{\frac{2k-d}2} \sin ( |y|-\alpha_{d,k} )\int e^{it\tau} \phi(\tau)   \big( \CI_4(\tau,y; \ep) -i \,\CI_3(\tau,y;\ep) \big) d\tau.	\]
Thanks to Lemmas \ref{lem_Ij} and \ref{lem_add}, if $c_1\ep^{-1} \le |y| \le c_2\ep^{-1}$, then we have the estimate
\begin{equation}\label{e:2}
 |\CJ_{k-1}^2(y,t)| \lesssim |y|^{\frac{2k-d}2} |\sin ( |y|- \alpha_{d,k} )|. 
\end{equation}
Moreover, since $|R(\rho|y|)|\lesssim (\rho|y|)^{-\frac32}$ for $|y|\gtrsim 1/\rho\sim 1$, we see from Lemma \ref{lem_Ij} that 
\begin{equation}\label{e:3}
\begin{aligned}
|\CJ_{k-1}^3 (y,t)| &\lesssim |y|^{\frac{2k-d}2-1} \int |\phi(\tau)| \big( |\wt\CI_1(\tau; \ep)|+|\wt\CI_2(\tau; \ep)| \big) d\tau \\
&\lesssim |y|^{\frac{2k-d}2-1} \log\ep^{-1}.
\end{aligned}
\end{equation}
Combining the estimates \eqref{e:1}, \eqref{e:2}, and \eqref{e:3}, if $|t|\ll 1$ and $y$ lies in the set
 \begin{equation}\label{e:set_lower}
	\mathfrak S := \left\{y\in \R^{d-1}\colon c_1\ep^{-1} \le |y| \le c_2\ep^{-1},\, |y|-\alpha_{d,k}\in 2\pi\Z+[-c_0, c_0] \right\}
\end{equation}
for a sufficiently small constant $c_0$, then
\begin{align*}
 |\CJ_{k-1}(y,t)| &\ge |\CJ_{k-1}^1(y,t)|-|\CJ_{k-1}^2(y,t)|-|\CJ_{k-1}^3(y,t)| \\
 & \gtrsim |y|^{\frac{2k-d}2} \big( \cos ( |y|- \alpha_{d,k}) - c |\sin ( |y|-\alpha_{d,k})| - c|y|^{-1} \log\ep^{-1} \big) \\
 &\gtrsim \ep^{\frac d2-k}.
 \end{align*}
Thus the proof of the estimate \eqref{Jk} is complete. 

Secondly, we prove the estimates \eqref{Jl}.  Since $|J_\nu(r)|\lesssim r^{-\frac12}$ for $r\gtrsim 1$, Lemma \ref{lem_Ij} yields, for $s>0$ and $y\in \mathfrak S$,
\begin{align*}
|\CJ_l(y,t) |&\lesssim |y|^{\frac{2-d}2+l} \int |\phi(\tau)| \big( |\wt\CI_1(\tau;\ep) |+|\wt\CI_2(\tau;\ep)| \big) d\tau\\
&\lesssim |y|^{\frac{2-d}2+l} \big(1 +\log\ep^{-1}\big) \lesssim \ep^{\frac d2-l-1-s}.
\end{align*}

To sum up, we have proved that if $|t|\ll 1$ and $y\in \mathfrak S$ (with $c_0$ small enough), then 
\[	|\wt m_\ep (D) f(y,t)| \gtrsim \ep^{\frac d2-k}, \ \ 0<\ep \ll1. 	\]
Therefore, for some constant $\tilde c>0$ small enough,
\[	\|\wt m_\ep (D)\|_{p\to q} \gtrsim \|\wt m_\ep (D) f\|_{L^q(\mathfrak S\times [-\tilde c,\tilde c])} \gtrsim \ep^{\frac d2-k} |\mathfrak S|^{\frac1q} \sim \ep^{\frac{d}2-k-\frac{d-1}q},	\]
and the proof of \eqref{eq_goal} is completed.
\end{proof}

\subsubsection*{Acknowledgement}
Jeong is grateful for support by the Open KIAS Center at Korea Institute for Advanced Study. She is also partially supported by NRF-2020R1F1A1A01048520. Kwon is supported by a KIAS Individual Grant (MG073702)  and NRF-2020R1F1A1A01073520. Lee is supported by NRF-2022R1A4A1018904.

\bibliographystyle{amsalpha}

\end{document}